\documentclass[11pt,a4paper,reqno]{amsart}
\usepackage[utf8]{inputenc}
\usepackage[T1]{fontenc}
\usepackage[UKenglish]{babel}
\usepackage[a4paper,margin=1in]{geometry}

\usepackage{amsmath, amsfonts, amssymb, amsthm, mathtools, indentfirst, bbm}
\usepackage[dvipsnames]{xcolor}
\usepackage[shortlabels]{enumitem}
\usepackage[languagenames,fixlanguage]{babelbib}
\usepackage{commath}
\usepackage{textcomp}
\usepackage{mathrsfs,dsfont}
\allowdisplaybreaks

\makeatletter
\def\namedlabel#1#2{\begingroup
    #2%
    \def\@currentlabel{#2}%
    \label{#1}\endgroup
}
\makeatother

\makeatletter \@addtoreset{equation}{section}

\makeatother

\newcommand{\cA}{\mathcal{A}}
\newcommand{\cC}{\mathcal{C}}
\newcommand{\cJ}{\mathcal{J}}
\newcommand{\cK}{\mathcal{K}}
\newcommand{\cV}{\mathcal{V}}
\newcommand{\cL}{\mathcal{L}}
\newcommand{\cT}{\mathcal{T}}
\newcommand{\cM}{\mathcal{M}}
\newcommand{\cD}{\mathcal{D}}

\usepackage{hyperref}

\newcommand{\Z}{\mathbb{Z}}

\newcommand{\ex}{\mathbf{E}}

\theoremstyle{plain}
\newtheorem{theorem}{Theorem}[section]
\newtheorem{corollary}[theorem]{Corollary}
\newtheorem{lemma}[theorem]{Lemma}
\newtheorem{proposition}[theorem]{Proposition}

\theoremstyle{definition}
\newtheorem{remark}[theorem]{Remark}

\newtheorem{definition}[theorem]{Definition}

\numberwithin{equation}{section}

\renewcommand\labelenumi{\textup{\alph{enumi})}}

\renewcommand\theenumi\labelenumi

\makeatletter\renewcommand{\p@enumii}{}\makeatother 

\makeatletter
\def\namedlabel#1#2{\begingroup
	#2%
	\def\@currentlabel{#2}%

	\label{#1}\endgroup
}

\newcounter{conum} 

\begin{document}

\title[IDS for subordinate Brownian motions in Poisson environment]{IDS for subordinate Brownian motions in Poisson random environment on nested fractals}
\author{Hubert Balsam, Kamil Kaleta, Mariusz Olszewski, Katarzyna Pietruska-Pa\l uba}

\address{H. Balsam \\ Faculty of Mathematics, Informatics and Mechanics, University of Warsaw
\\ ul. Banacha 2, 02-097  Warszawa, Poland}
\email{hubert.balsam@gmail.com}

\address{K. Kaleta \\ Faculty of Pure and Applied Mathematics, Wroc{\l}aw University of Science and Technology, Wyb. Wyspia\'nskiego 27, 50-370 Wroc{\l}aw, Poland}
\email{kamil.kaleta@pwr.edu.pl}

\address{M. Olszewski \\ Faculty of Pure and Applied Mathematics, Wroc{\l}aw University of Science and Technology, Wyb. Wyspia\'nskiego 27, 50-370 Wroc{\l}aw, Poland}
\email{mariusz.olszewski@pwr.edu.pl}

\address{K. Pietruska-Pa{\l}uba \\ Institute of Mathematics \\ University of Warsaw
\\ ul. Banacha 2, 02-097 Warszawa, Poland}
\email{kpp@mimuw.edu.pl}

\maketitle

\begin{abstract}
We establish the Lifshitz singularity of the integrated density of states (IDS) 
for random Schr\"odinger operators
\[
H^{\omega} = \phi(-\mathcal{L}) + V^{\omega}
\]
on planar unbounded nested fractals with the Good Labeling Property. Here, 
$\mathcal{L}$ is the Laplacian on the fractal, $\phi$ is an operator monotone 
function with mild regularity, and $V^{\omega}$ is a Poissonian random potential 
with a sufficiently regular profile. The main novelty of our work lies in showing that the study of $V^{\omega}$ can be effectively reduced to the analysis of certain alloy-type potential, where the sites are no longer lattice points as in the classical $\mathbb{Z}^d$ case, but fractal complexes. 
This observation enables us to apply an approach, new in the setting of Poissonian random fields, which allows us to treat a broad class of Bernstein functions $\phi$. 
In particular, it covers the case $\phi(\lambda)=(\lambda+m^{d_w/\vartheta})^{\vartheta/d_w}-m$, 
$\vartheta \in (0,d_w)$, $m>0$, corresponding to relativistic models, which were previously unattainable on fractals by known methods.


\bigskip
\noindent
2020 {\it MS Classification Primary}:\ 82B44, 28A80, 60K37; Secondary: 47D08, 60J45, 60J57.
\end{abstract}

\footnotetext{Research was supported by the National Science Centre, Poland, grant no.\ 2019/35/B/ST1/02421.}

\section{Introduction}

Let $X = (X_t)_{t \geq 0}$ be a subordinate Brownian motion with values in a planar unbounded simple nested fractal 
$\cK^{\langle \infty \rangle}$ (USNF) satisfying the Good Labeling Property (GLP). The infinitesimal generator of $X$ is the operator $-\phi(-\cL)$, where $\cL$ denotes the canonical Laplacian on $\cK^{\langle \infty \rangle}$, i.e., the generator of the Brownian motion $Z = (Z_t)_{t \geq 0}$ on this fractal space, and $\phi$ is a complete Bernstein function satisfying $\phi(0+)=0$ \cite{bib:SSV}. The process $X$ is obtained from $Z$ by a random time change. 
More precisely, $X_t = Z_{S_t}$, $t \geq 0$, where $S = (S_t)_{t \geq 0}$ is a subordinator with Laplace transform $\exp(-t \phi(\lambda))$, $\lambda>0$. 

We assume that the process $X$ evolves in a random environment on $\cK^{\langle \infty \rangle}$, 
generated by a Poisson point process independent of $X$. The effect of this random medium is described by a potential of the form
\begin{equation}
\label{eq:potendef_intro}
V^{\omega}(x) 
:= \int_{\cK^{\langle \infty \rangle}} 
W(x,y)\,\mu^{\omega}(\mathrm{d}y),
\end{equation}
where $\mu^{\omega}$ is the counting measure associated with a Poisson point process on $\cK^{\langle \infty \rangle}$, 
and $W$ is a nonnegative potential profile depending on two variables. The Poisson point process has intensity measure 
$\nu\, \mathfrak{m}(\mathrm{d}y)$, where $\mathfrak{m}$ denotes the normalized Hausdorff measure on the fractal, and $\nu>0$ is a fixed parameter.
The random field $V^{\omega}$ given by \eqref{eq:potendef_intro} will be called a {\em fractal Poisson-type potential}.

The main goal of this paper is to investigate  spectral properties 
of the random Schr\"odinger operator
\[
H^{\omega} = \phi(-\cL) + V^{\omega},
\quad \text{acting in} \quad L^2(\cK^{\langle \infty \rangle}, \mathfrak{m}).
\]
In particular, we study the asymptotic behaviour of the \emph{integrated density of states} (IDS) of $H^{\omega}$ 
under mild regularity assumptions on $\phi$ and the potential profile $W$. The IDS is in this case a non-random measure supported on $[0,\infty)$. In the fractal setting, it is defined as the vague limit of empirical measures counting the eigenvalues of the finite-volume operators obtained by restricting $H^{\omega}$ to $L^2(\cK^{\langle M \rangle}, \mathfrak{m})$, normalized by $\mathfrak{m}(\cK^{\langle M \rangle})$. Here, $\cK^{\langle M \rangle}$ denote the fractal complexes that approximate $\cK^{\langle \infty \rangle}$ as $M \to \infty$. Throughout, we denote the IDS by $\Lambda$. For simplicity, we use the same symbol for  its distribution function, $\Lambda(\lambda):=\Lambda([0,\lambda])$, $\lambda > 0$.

The following theorem is the main result of the paper. It shows that,
under assumption \eqref{B} on the Bernstein function $\phi$ and
assumptions \eqref{W1}, \eqref{W2} on the single-site profile $W$,
the IDS of the Schr\"odinger operator $H^{\omega}$ exhibits
Lifshitz-type behaviour at low energies.

\begin{theorem} \label{th:main_IDS}
 Assume \eqref{B}, \eqref{W1} and \eqref{W2}. Then for every $\nu_0>0$ there exist constants $C_1, C_2>0$ such that for every $\nu \geq \nu_0$
\[
- C_1 \nu \le
\liminf_{\lambda \searrow 0}
\lambda^{\frac{d}{\alpha}} \log \Lambda(\lambda)
\qquad \text{and} \qquad
\limsup_{\lambda \searrow 0}
\lambda^{\frac{d}{\alpha}} \log \Lambda(\lambda)
\le - C_2 \nu.
\]
\end{theorem}

We would like to emphasize that in this non-smooth fractal setting 
the very existence of the IDS is already a non-trivial issue. 
In the present framework, however, it can be established by following 
a general method,  developed in \cite{bib:KaPP2} for the planar Sierpi\'nski gasket. 
This approach requires a sufficiently regular periodic structure of the fractal,  
which we refer to as the Good Labeling Property (GLP) \cite{bib:KOP} (see also \cite{bib:NO,bib:O}). 
In the present paper, we take the existence of the IDS for granted. 
However, for the reader's convenience and completeness, we include it as Theorem \ref{thm:IDS}
and briefly discuss the main ideas of the proof in Section \ref{sec:IDS_ex}. 
A concise outline of the argument, together with precise references, 
is provided in Appendix \ref{sec:appA}.

The results of this paper fit into the rapidly developing area of  analysis and probability on fractals, which remains a very active field of research. We refer, for example, to the work of Kigami \cite{bib:Kig4}, the series of papers by Alonso-Ruiz et al. \cite{bib:ARBCRST1,bib:ARBCRST2,bib:ARBCRST3,bib:ARBCRST4}, which focus in particular on function spaces on fractals, and to the work of Baudoin et al.~\cite{bib:FB-LC-CHH-CO-ST-JW} on the Parabolic Anderson Model. Fractal spaces and domains with fractal boundaries also play an increasingly important role in modeling and applications, including theory of PDE's (Chen et al.\ in \cite{bib:CHT} and Dekkers et al. \cite{bib:DRT}), information theory  (Akkermans et al.\ \cite{bib:ACDRT}), magnetostatics problems (Hinz and Teplyaev \cite{bib:HT}), acoustics theory (Hinz et al.\ \cite{bib:HRT}), theory of random walks on these sets (e.g.\ Kumagai and  Nakamura \cite{bib:KN}). See also the paper of Kumagai \cite{bib:Kum7} for an overview on diffusions on disordered  media.

Our framework, where the subordinate Brownian motion $X$ evolves in a Poisson random medium on a fractal, 
can be viewed as a model of a particle moving in a perfectly ordered material with impurities, defects, or dislocations — that is, in a disordered medium. Such perfect order is described by the USNF $\cK^{\langle \infty \rangle}$ with the GLP, 
which has a self-similar and well-defined periodic structure, while the disorder is introduced by the Poissonian random field $V^{\omega}$.
Defects of the structure are modeled by a random cloud of sites 
$\{y_i(\omega)\}$ generated by the Poisson point process, located irregularly and independently of the geometry of 
$\cK^{\langle \infty \rangle}$. The effect of a single Poissonian site $y_i(\omega)$ on a particle 
at position $x$ is described by a single-site potential $W(x,y_i(\omega))$, and the total potential is given by
\begin{align} \label{eq:poiss_pot_heur}
V^{\omega}(x) = \sum_{i} W(x,y_i(\omega)) = \int_{\cK^{\langle \infty \rangle}} 
W(x,y)\,\mu^{\omega}(\mathrm{d}y),
\end{align}
as introduced above. The Schr\"odinger operator
\[
H^{\omega} = \phi(-\cL) + V^{\omega}
\]
represents the Hamiltonian of the system. Its kinetic term, $\phi(-\cL)$, which is the negative of the generator of the process $X$, can come from a wide range of operators, depending on the choice of the complete Bernstein function $\phi$.
Assuming that the motion of a quantum particle follows, at least statistically, the distributional properties of the process $X$ with generator $-\phi(-\cL)$, our framework covers not only the non-relativistic case with $\phi(\lambda)=\lambda$ (corresponding to Brownian motion on $\cK^{\langle \infty \rangle}$) and $\phi(\lambda)=\lambda^{\alpha/d_w},$ $\alpha\in (0,d_w)$ (corresponding to  stable processes on $\cK^{\langle\infty\rangle})$, but also relativistic models with 
\[
\phi(\lambda) = (\lambda + m^{d_w/\vartheta})^{\vartheta/d_w} - m, \quad \vartheta \in (0,d_w), \ m>0
\]
(corresponding to relativistic stable processes on $\cK^{\langle \infty \rangle}$). Such jump processes can travel very large distances in a short time via a combination of small jumps rather than classical diffusion, making them a reasonable model for a particle moving at very high velocity. 

This interpretation, in particular the fact that the operator $H^{\omega}$ represents the energy of the system, shows that 
understanding of the spectral properties of $H^{\omega}$, including those of the IDS, is of fundamental importance (see the monographs of Carmona and Lacroix~\cite{bib:Car-Lac} and Stollmann~\cite{bib:Stoll}). In particular, the Lifschitz singularity of the IDS of a random Schr\"{o}dinger operator is closely related to the phenomenon 
of spectral localisation, see, e.g., Bourgain and Kenig~\cite{bib:Bou-Ken}, Combes and Hislop~\cite{bib:Com-His},
Germinet, Hislop and Klein~\cite{bib:Ger-His-Kle}, Kirsch and Veseli\'{c}~\cite{bib:KV}, Klopp~\cite{bib:Klopp1,bib:Klopp2}, and references therein. Note that Lifschitz tails have been extensively studied in both the Euclidean and the integer-lattice settings for various types of random potentials.  
For results concerning Schr\"odinger operators with Laplace kinetic terms, we refer the reader to 
Benderskii and Pastur~\cite{bib:BP}, 
Friedberg and Luttinger~\cite{bib:FL},
Fukushima~\cite{bib:F}, 
Fukushima, Nagai and Nakao~\cite{bib:FNN}, 
Kirsch and Martinelli~\cite{bib:KM1,bib:KM2}, 
Kirsch and Simon~\cite{bib:KS}, 
Kirsch and Veseli\'c~\cite{bib:KV}, 
Luttinger~\cite{bib:Lut}, 
Mezincescu~\cite{bib:Mez},
Nagai~\cite{bib:Nag}, 
Nakao~\cite{bib:Nak}, 
Pastur~\cite{bib:Pas}, 
Romerio and Wreszinski~\cite{bib:RW}, 
and Simon~\cite{bib:Sim}. 
For non-local kinetic terms, see 
Okura~\cite{bib:Okura},
Gebert and Rojas-Molina~\cite{bib:GRM,bib:GRM2}, and
Kaleta and Pietruska-Pa{\l}uba~\cite{bib:KaPP4,bib:KaPP3}.

In this context, we also refer to the seminal paper of Donsker and Varadhan \cite{bib:DV}, and Sznitman \cite{bib:Szn1}, concerned with Poissonian obstacles and potentials in the Euclidean space.

On irregular structures such as fractals, results concerning the Lifschitz tail of the IDS are scarce. Early work by Pietruska-Pa{\l}uba~\cite{bib:KPP-PTRF} treated Brownian motion with Poissonian killing obstacles on the planar Sierpiński gasket, while Shima~\cite{bib:Sh} extended this to general nested fractals with finite-range Poisson potentials. Later, Kaleta and Pietruska-Pa{\l}uba~\cite{bib:KaPP} analysed operators of the form $\phi(-\cL)+V^\omega$, corresponding to a class of subordinate Brownian motions and more general Poisson potentials on the planar Sierpiński gasket. Most recently, they established the Lifschitz tail for subordinate Brownian motions on USNF with the GLP, as considered in the present paper, but for a different type of random fields, namely fractal alloy-type potentials \cite{bib:HB-KK-MO-KPP}.

To the best of our knowledge, the paper~\cite{bib:KaPP} was the most recent work on Poisson random media on fractals. However, its scope was essentially limited, both in terms of the class of subordinate processes (equivalently, the kinetic terms $\phi(-\cL)$) and the underlying state space, being restricted to the planar Sierpiński gasket. The reason lies in the method of the proof, a variant of the coarse-graining technique introduced by Sznitman \cite{bib:Szn1}, which applies only to those subordinate processes that, after suitable reductions, can be brought to the $\alpha$-stable case. This approach, although technically involved, is particularly well suited to processes possessing scaling properties. Consequently, that paper excluded non-local models with $\phi(\lambda) \approx \lambda$ as $\lambda \searrow 0$, which are arguably the most interesting from the perspective of mathematical physics, such as the relativistic models with mass $m>0$ mentioned above.

This was exactly the starting point of the present project:\ over the years, we have been seeking for an appropriate argument that would allow us to tackle this problem.  The solution that arose only recently is based on an idea reducing the Poissonian random potential 
to a potential structurally close to one of alloy-type. 
In other words, the Poissonian random environment on a fractal can be effectively transformed into a random medium with  certain alloy-type structure. This may seem somewhat surprising, as the Poisson random potential $V^{\omega}$ in \eqref{eq:poiss_pot_heur} is structurally quite different from the standard alloy-type potential 
\[
V^{\omega}_{\mathrm{alloy}}(x) := \sum_{v \in \cV_0^{\langle \infty \rangle}} \xi_v(\omega)\, W(x,v),
\]
which was recently introduced in \cite{bib:HB-KK-MO-KPP} as a fractal analog of the classical alloy-type model in $\mathbb{Z}^d$. 
Indeed, for $V^{\omega}$, the disorder arises from the random locations of the Poisson sites $y_i(\omega)$, whereas for $V^{\omega}_{\mathrm{alloy}}$ the sites are arranged regularly according to the geometry of the fractal, and the randomness is carried by the single-site random variables $\xi_v(\omega)$, which form an i.i.d.\ family. In this approach, we introduce the potential
\begin{align} \label{eq:alloy-complex}
\overline{V}^{\,\omega}(x) := \sum_{\Delta \in \cT_0} \xi_{\Delta}(\omega)\,\overline{W}(x,\Delta),
\end{align}
where the i.i.d.\ single-site random variables $\xi_{\Delta}(\omega)$ are associated with fixed size fractal complexes, and 
$\overline{W} : \cK^{\langle \infty \rangle} \times \cT_0 \to [0,\infty)$ is the corresponding single-site potential; see Remark \ref{rem:complex_alloy} for a further discussion. In our present setup, the distribution of the $\xi_{\Delta}$'s is induced by $\mu^{\omega}$.

To the best of our knowledge, this concept is novel. It constitutes a key step in the proof of Theorem \ref{th:main}, where the  the upper bound for the Laplace transform of the IDS is proven. 
The main advantage of this step is that it allows us to apply to $\overline{V}^{\,\omega}$ the powerful method developed recently in \cite{bib:KaPP3} and further applied in \cite{bib:HB-KK-MO-KPP}. 
This method is functional-analytic in nature, relying on the application of Temple's inequality. 
It has been highly effective for Schr\"odinger operators with standard alloy-type random potentials, i.e., with single-sites located on lattices, but we have found a way to adapt it to the setting of the new alloy-type potentials of the form \eqref{eq:alloy-complex}. This new approach allows us to extend the results of \cite{bib:KaPP} to a broad class of subordinate processes (or kinetic operators) with $\phi$ satisfying assumption \eqref{B}, with potential profiles $W$ of bounded support, on USNFs possessing the GLP.
We note that the idea of using Temple's inequality in the context of alloy-type potentials originates from the influential works of Simon \cite{bib:Sim}, and Kirsch and Simon \cite{bib:KS}.

The paper is organized as follows.
Section \ref{sec:prel} contains the necessary preliminary material needed for the proofs of the main theorems. This includes essentials on nested fractals, the good labeling property, stochastic processes on fractal sets, their semigroups and generators, and random Schr\"odinger operators. We recall that the class of subordinate processes and the class of fractals studied in this paper coincide with those considered in our recent work \cite{bib:HB-KK-MO-KPP}, where alloy-type potentials on fractals were considered. For the reader’s convenience, we repeat the necessary definitions and results here, and occasionally refer to the cited paper for further technical details.
In Section \ref{subsec:FK}, we introduce a class of fractal random Poisson potentials and their periodizations, and we prove that the random potentials under consideration belong to the Kato classes of the respective processes (Proposition \ref{pro:katoclass}). The existence of the IDS is further discussed in Section \ref{sec:IDS_ex}.
Section \ref{sec:LS_general} contains the proof of the upper bound for the Laplace transform of the IDS (Theorem \ref{th:main}), which constitutes the central part of the paper. The lower bound for the IDS (Theorem \ref{th:main-2}) and the proof of the Lifschitz singularity (Theorem \ref{th:main_IDS}) are given in Section \ref{sec:LS_bounded}. Appendix \ref{sec:appA} provides additional details and references concerning the existence of the IDS.

\section{Preliminaries} \label{sec:prel}

\subsection{Unbounded (planar) simple nested fractals} \label{sec:usnf}

In this section, we recall several basic notions concerning planar nested fractals that will be used throughout the paper. For consistency, we adopt the approach developed in our earlier work \cite{bib:HB-KK-MO-KPP}. This framework is also consistent with earlier studies on nested fractals \cite{bib:Kum1,bib:Lin,bib:kpp-sausage,bib:kpp-sto}.

Consider a family of similitudes $\{\Psi_i\}_{i=1}^N$ acting on $\mathbb{R}^2$, all having the same scaling factor $L>1$ and the same isometric part $U$. Thus, each mapping is given by $\Psi_i(x) = (1/L) U(x) + \nu_i$,
where $\nu_i \in \mathbb{R}^2$ for $i \in \{1, \dots, N\}$. Without loss of generality, we assume that $\nu_1 = 0$.
There exists a unique nonempty compact set $\mathcal{K}^{\langle 0 \rangle}$, referred to as the \emph{fractal generated by the system} $(\Psi_i)_{i=1}^N$, satisfying the invariance relation
\[
\mathcal{K}^{\langle 0 \rangle}
=
\bigcup_{i=1}^{N} \Psi_i\big( \mathcal{K}^{\langle 0 \rangle} \big).
\]
Since $L>1$, each similitude admits a unique fixed point. Consequently, the transformations $\Psi_1, \dots, \Psi_N$ have exactly $N$ fixed points in total.  We denote by $F$ the set consisting of these fixed points.
A point $x \in F$ is called an \emph{essential fixed point} if there exist another point $y \in F$ and two distinct similitudes $\Psi_i$ and $\Psi_j$ such that $\Psi_i(x) = \Psi_j(y)$. 
The collection of all essential fixed points of the transformations $\Psi_1, \dots, \Psi_N$ is denoted by $\mathcal{V}_{0}^{\langle 0 \rangle}$.

The set $\mathcal{K}^{\left\langle 0\right\rangle}$ is called a (planar) \emph{nested fractal} if the following conditions are met.
\begin{enumerate}
\item $k:=\# \cV_{0}^{\left\langle 0\right\rangle} \geq 2.$
\item (Open Set Condition) There exists an open set $U \subset \mathbb{R}^2$ such that for $i\neq j$ one has\linebreak $\Psi_i (U) \cap \Psi_j (U)= \emptyset$ and $\bigcup_{i=1}^N \Psi_i (U) \subseteq U$.
\item (Nesting) $\Psi_i\left(\mathcal{K}^{\left\langle 0 \right\rangle}\right) \cap \Psi_j \left(\mathcal{K}^{\left\langle 0 \right\rangle}\right) = \Psi_i \left(\cV_{0}^{\left\langle 0\right\rangle}\right) \cap \Psi_j \left(\cV_{0}^{\left\langle 0\right\rangle}\right)$ for $i \neq j$.
\item (Symmetry) For $x,y \in \cV_{0}^{\left\langle 0\right\rangle},$ let $S_{x,y}$ denote the symmetry with respect to the line bisecting the segment $\left[x,y\right]$. Then
\begin{equation*}
\forall i \in \{1,...,M\} \ \forall x,y \in \cV_{0}^{\left\langle 0\right\rangle} \ \exists j \in \{1,...,M\} \ S_{x,y} \left( \Psi_i \left(\cV_{0}^{\left\langle 0\right\rangle} \right) \right) = \Psi_j \left(\cV_{0}^{\left\langle 0\right\rangle} \right).
\end{equation*}
\item (Connectivity) On the set $\cV_{-1}^{\left\langle 0\right\rangle}:= \bigcup_i \Psi_i \left(\cV_{0}^{\left\langle 0\right\rangle}\right)$ we define graph structure $E_{-1}$ as follows:\\
$(x,y) \in E_{-1}$ if and only if $x, y \in \Psi_i\left(\mathcal{K}^{\left\langle 0 \right\rangle}\right)$ for some $i$.
Then the graph $(\cV_{-1}^{\left\langle 0\right\rangle},E_{-1} )$ is required to be connected.
\end{enumerate}

Since for $k = 2$ the set $\mathcal{K}^{\langle 0 \rangle}$ reduces to a line segment joining two essential fixed points, we assume throughout that $k \geq 3$. In this case, the elements of $\mathcal{V}_{0}^{\langle 0 \rangle}$ form the set of vertices of a regular polygon \cite[Proposition 2.1]{bib:KOP}.

We next define the following objects.
\begin{definition} Let $M\in\mathbb{Z}.$
\begin{itemize}
\item[(1)]
\begin{equation} \label{eq:Kn}
\mathcal{K}^{\left\langle M\right\rangle} = L^M \mathcal{K}^{\left\langle 0\right\rangle}
  = \Psi^{-M}_1(\mathcal K^{\langle 0\rangle}),
\end{equation}
 is the fractal of size $L^M$
(the $M$-complex  attached to $(0,0)$, see \eqref{eq:Mcompl} below).

\item[(2)]
\begin{equation} \label{eq:Kinfty}
\mathcal{K}^{\left\langle \infty \right\rangle} = \bigcup_{M=0}^{\infty} \mathcal{K}^{\left\langle M\right\rangle}.
\end{equation}
is the unbounded simple nested fractal (USNF).
\item[(3)] $M$-complex: \label{def:Mcomplex}
every set $\Delta_M \subset \mathcal{K}^{\left\langle \infty \right\rangle}$ of the form
\begin{equation} \label{eq:Mcompl}
\Delta_M  = \mathcal{K}^{\left\langle M \right\rangle} + \nu_{\Delta_M},
\end{equation}
where $\nu_{\Delta_M}=\sum_{j=M+1}^{J} L^{j} \nu_{i_j},$ for some $J \geq M+1$, $\nu_{i_j} \in \left\{\nu_1,...,\nu_N\right\}$, is called an \emph{$M$-complex}. The family of all $M$-complexes in $\mathcal K^{\langle \infty \rangle}$ will be denoted by $\mathcal{T}_M$.
\item[(4)] Vertices of the $M$-complex \eqref{eq:Mcompl}: the set $\cV\left(\Delta_M\right) =L^M\cV_0^{\langle 0 \rangle}+\nu_{\Delta_M}= L^{M} \mathcal V^{\left\langle 0 \right\rangle}_0 + \sum_{j=M+1}^{J} L^{j} \nu_{i_j}$.
\item[(5)] Vertices of $\mathcal{K}^{\left\langle M \right\rangle}$:
$$
\mathcal V^{\left\langle M\right\rangle}_{M} = \cV\left(\mathcal{K}^{\left\langle M \right\rangle}\right) = L^M \mathcal V^{\left\langle 0\right\rangle}_{0}.
$$
\item[(6)] Vertices of all $M$-complexes inside an $(M+m)$-complex (defined recursively for $m>0$):
$$
\cV_M^{\langle M+m\rangle}= \bigcup_{i=1}^{N} \cV_M^{\langle M+m-1\rangle} + L^{M+m} \nu_i.
$$
\item[(7)] Vertices of all 0-complexes inside the unbounded nested fractal:
$$
\mathcal V^{\left\langle \infty \right\rangle}_{0} = \bigcup_{M=0}^{\infty} \mathcal V^{\left\langle M\right\rangle}_{0}.
$$
\item[(8)] Vertices of $M$-complexes from the unbounded fractal:
$$
\mathcal V^{\left\langle \infty \right\rangle}_{M} = L^{M} \mathcal V^{\left\langle \infty \right\rangle}_{0}.
$$
\item[(9)]$\cC_M(x)$:  the union of all $M-$complexes containing $x$:\\
-- the unique $M-$complex when $x\notin \cV_M^{\langle \infty \rangle},$\\
-- the union of all $M-$complexes intersecting at $x$ when $x\in \cV_M^{\langle \infty\rangle}.$ \\
The number of such complexes is denoted by ${\rm rank}_M\,(x)$ and may depend on $M.$
\item[(10)] $M$-complexes in $\cK^{\langle M+1 \rangle}$:
$$
\Delta_{M,i} = \cK^{\langle M \rangle} + L^M \nu_i, \quad i= 1,...,N.
$$
\end{itemize}
\end{definition}

The fractal (Hausdorff) dimension of $\mathcal{K}^{\langle \infty \rangle}$ is $d = \frac{\log N}{\log L}$. The Hausdorff measure in dimension $d$ supported on $\mathcal{K}^{\langle \infty \rangle}$ will be denoted by $\mathfrak{m}$, and it is normalized so that $\mathfrak{m}(\mathcal{K}^{\langle 0 \rangle}) = 1$. This measure serves as a \emph{uniform} measure on $\mathcal{K}^{\langle \infty \rangle}$.

\begin{definition}\label{def:graph-distance}
For $M \in \mathbb Z$ and $x,y \in \mathcal{K}^{\left\langle \infty \right\rangle}$ let
\begin{equation}
\label{eq:graphmetric}
d_M (x,y):= \left\{ \begin{array}{ll}
0, & \textrm{if } x=y ;\\[1mm]
1, & \textrm{if there exists } \Delta_M \in \mathcal{T}_M \textrm{ such that } x,y \in \Delta_M \textrm{ and } x \neq y ;\\[1mm]
n>1, & \textrm{if there does not exist } \Delta_M \in \mathcal{T}_M \textrm{ such that } x,y \in \Delta_M \textrm{ and } \\ 
& n \textrm{ is the smallest number for which there exists a chain } \\
& \Delta_M^{(1)}, \Delta_M^{(2)}, ..., \Delta_M^{(n)} \in \mathcal{T}_M \textrm{ such that } x \in \Delta_M^{(1)},
 y \in \Delta_M^{(n)} \\ &  \textrm{and } \Delta_M^{(i)} \cap \Delta_M^{(i+1)} \neq \emptyset \textrm{ for  }1 \leq i \leq n-1.
\end{array} \right.
\end{equation}
\end{definition}
The set $\mathcal{C}_M(x)$ defined above is simply the ball of radius $1$ centered at $x$ in the metric $d_M$.

Let
\begin{equation}\label{eq:def-r0}
r_0 := \max \{\mathrm{rank_0}(v) : v \in \mathcal{V}_0^{\langle \infty \rangle}\}.
\end{equation}
For $k \geq 4$ we have $r_0 = 2$, whereas for $k=3$ it holds that $r_0 \in \{2,3\}$. In the case of the Sierpi\'nski gasket $r_0 = 2$, while \cite[Example 3.2]{bib:KOP} illustrates the possibility of  $r_0 = 3$.

We also need the estimate on the number of points from the $0-$grid     inside $\mathcal K^{\langle M\rangle},$ for $M\in\mathbb Z_+,$ i.e.\ the cardinality of $\cV_0^{\langle M\rangle}.$
Denote this number by $k_0^{\langle M\rangle}.$
We readily see that there exists a constant $C_0>1$ such that
\begin{equation}\label{eq:number-points}
L^{ Md}\leq \# \mathcal V^{\langle M\rangle}_0 = k_0^{\langle M\rangle} \leq C_0 L^{ Md}.
\end{equation}

\subsection{Good Labeling Property and canonical projection} \label{sec:glp}
In this section, we recall the concept of a good labeling of vertices, as introduced in our previous work \cite{bib:KOP}.

Let $\mathcal{A} := \{a_1, a_2, \dots, a_k\}$ be an alphabet of $k$ symbols, where $k = \# \mathcal{V}_0^{\langle 0 \rangle} \geq 3$. The elements of $\mathcal{A}$ are called \emph{labels}. A \emph{labeling function of order $M \in \mathbb{Z}$} is any map
\[
l_M : \mathcal{V}_M^{\langle \infty \rangle} \to \mathcal{A}.
\]
Recall that the vertices of each $M$-complex $\Delta_M$ form the vertices of a regular polygon with $k$ vertices. In particular, there exist exactly $k$ rotations about the barycenter of $\mathcal{K}^{\langle M \rangle}$ that map $\mathcal{V}_M^{\langle M \rangle}$ onto itself. We denote these rotations by $\{R_1, \dots, R_k\} =: \mathcal{R}_M$, ordered so that $R_i$ rotates by an angle of $\frac{2 \pi i}{k}$ for $i = 1, \dots, k$.

\begin{definition}[\textbf{Good labeling function of order $M$}]
\label{def:glp} Let $M \in \mathbb{Z}$.  A function $\ell_M: \mathcal V^{\left\langle \infty \right\rangle}_{M} \to \cA$  is called a \emph{good labeling function of order $M$} if the following conditions are met.
\begin{itemize}
\item[(1)] The restriction of $\ell_M$ to $\mathcal V^{\left\langle M \right\rangle}_{M}$ is a bijection onto $\cA$.
\item[(2)] For every $M$-complex $\Delta_M$ represented as
$$
\Delta_M  = \mathcal{K}^{\left\langle M \right\rangle} +\nu_{\Delta_M},
$$
where $\nu_{\Delta_M}=  \sum_{j=M+1}^{J} L^{j} \nu_{i_j},$  with some $J \geq M+1$ and $\nu_{i_j} \in \left\{\nu_1,...,\nu_N\right\}$ (cf. Def. \ref{def:Mcomplex}), there exists a rotation $R_{\Delta_M} \in \mathcal{R}_M$ such that
\begin{align}\label{eq:rot}
\ell_M(v)=\ell_M\left(R_{\Delta_M}\left(v -\nu_{\Delta_M}\right)\right) , \quad v \in \cV\left(\Delta_M\right).
\end{align}
\end{itemize}
\end{definition}
A USNF $\mathcal{K}^{\langle \infty \rangle}$ is said to have the \emph{good labeling property of order $M$} if there exists a good labeling function of order $M$.

Due to the self-similar structure of $\mathcal{K}^{\langle \infty \rangle}$, the \emph{good labeling property of order $M$} for some $M \in \mathbb{Z}$ is equivalent to the same property at any other order $M' \in \mathbb{Z}$. This motivates the following general definition.

\begin{definition}[\textbf{Good labeling property}] \label{def:glp_gen} A USNF $\mathcal{K}^{\left\langle \infty \right\rangle}$ is said to have the \emph{good labeling property (GLP in short)} if it has the good labeling property of order $M$ for some $M \in \mathbb{Z}$.
\end{definition}

In other words, the fractal $\mathcal{K}^{\langle \infty \rangle}$ has the GLP if and only if the vertices in $\mathcal{V}_M^{\langle \infty \rangle}$ can be labeled in such a way that each $M$-complex contains the complete set of labels and the order of labels is preserved across complexes.

It is known that the fractal has the GLP when $k$ is prime or of the form $k = 2^n$. For other values of $k$, one can construct examples both with and without the GLP, depending on the arrangement of $0$-complexes within $1$-complexes. We refer the reader to \cite{bib:KOP, bib:NO} for a detailed analysis and examples. 

\medskip

For an unbounded fractal $\mathcal{K}^{\langle \infty \rangle}$ with the GLP, we define a projection map $\pi_{M}$ from $\mathcal{K}^{\left\langle \infty \right\rangle}$ onto the primary $M$-complex $\mathcal{K}^{\left\langle M \right\rangle}$ by
\begin{equation}\label{eq:piem}
\pi_M(x) := R_{\Delta_M}\bigl(x - \nu_{\Delta_M}\bigr), \quad x \in \mathcal{K}^{\langle \infty \rangle},
\end{equation}
where $\Delta_M = \mathcal{K}^{\langle M \rangle} + \nu_{\Delta_M} = \mathcal{K}^{\langle M \rangle} + \sum_{j=M+1}^{J} L^j \nu_{i_j}$ is the $M$-complex containing $x$, and $R_{\Delta_M} \in \mathcal{R}_M$ is the unique rotation determined by \eqref{eq:rot}. More precisely:

\begin{itemize}
\item[(1)] If $x \notin \mathcal{V}_M^{\langle \infty \rangle}$, we set $\Delta_M = \mathcal{C}_M(x)$, i.e.\ the unique $M$-complex containing $x$.
\item[(2)] If $x \in \mathcal{V}_M^{\langle \infty \rangle}$, any $M$-complex containing $x$ can be chosen; thanks to the GLP, the projection does not depend on this choice.
\end{itemize}

The restriction of $\pi_M$ to any $M$-complex $\Delta_M$ is a bijection, so its inverse
\[
\widetilde{\pi}_{\Delta_M} := \bigl(\pi_M\mid_{\Delta_M}\bigr)^{-1}
\]
is well defined and given by
\[
\widetilde{\pi}_{\Delta_M}(x) = R_{\Delta_M}^{-1}(x) + \nu_{\Delta_M}, \quad x \in \mathcal{K}^{\langle M \rangle}.
\]
More generally, we can project onto any $M$-complex $\Delta_M$ by setting
\begin{equation}\label{eq:uprojection}
\pi_{\Delta_M}(x) := \widetilde{\pi}_{\Delta_M}\bigl(\pi_M(x)\bigr), \quad x \in \mathcal{K}^{\langle \infty \rangle}.
\end{equation}
Clearly, for the primary $M$-complex we have $\pi_{\mathcal{K}^{\langle M \rangle}} = \pi_M$, since $\widetilde{\pi}_{\mathcal{K}^{\langle M \rangle}} = \mathrm{Id}$.

\subsection{Free and reflected subordinate Brownian motions}

\subsubsection{Free Brownian motion} 
The theory of Dirichlet forms (see \cite{bib:Fuk2, bib:Kum1}) provides a convenient framework for defining Brownian motion on simple nested fractals $\mathcal K^{\langle\infty\rangle}$ \cite{bib:Kus2,bib:Lin}. The resulting process $(Z_t,\mathbf P^x)_{t\ge 0,\; x\in\mathcal K^{\langle\infty\rangle}}$ is a symmetric strong Markov process with continuous paths. 
Its transition densities $g(t,x,y)$ (with respect to the $d$-dimensional Hausdorff measure $\mathfrak{m}$) satisfy the following sub-Gaussian estimates \cite[Theorems 5.2, 5.5]{bib:Kum1}:\
\begin{multline}
\label{eq:kum}
C_{1}\, t^{-d_s/2} 
\exp\!\left(
- C_{2}
\left(
\frac{|x-y|^{d_w}}{t}
\right)^{\!\frac{1}{d_J-1}}
\right)
\le g(t,x,y)
\\
\le
C_{3}\, t^{-d_s/2} 
\exp\!\left(
- C_{4}
\left(
\frac{|x-y|^{d_w}}{t}
\right)^{\!\frac{1}{d_J-1}}
\right),
\qquad
t>0,\;\; x,y\in \mathcal K^{\langle\infty\rangle}.
\end{multline}
Here $d_w$ denotes the walk dimension of $\mathcal K^{\langle\infty\rangle}$, 
$d_s = 2d/d_w$ its spectral dimension, and $d_J>1$ the so-called \emph{chemical exponent}. 
These bounds hold under mild structural assumptions, which are always satisfied in our case 
(see the comments following \cite[Lemma A.2]{bib:KOP}). The densities $g(t,x,y)$ are jointly continuous, symmetric and satisfy the scaling property
\[
g(t,x,y)
=
L^{d}\, g(L^{d_w}t, Lx, Ly),
\qquad
t>0.
\]

\subsubsection{Reflected Brownian motion on $\mathcal K^{\langle M \rangle}$.}
This process has been defined in \cite{bib:KOP}.
When $\mathcal K^{\langle\infty\rangle}$ is a fractal  which has the GLP, then for an arbitrary $M\in\mathbb Z$ the reflected Brownian motion $(Z_t^M, \mathbf{P}^{x}_{M})_{t \geq 0, \, x \in \mathcal{K}^{\left\langle M\right\rangle}}$ with values in $\mathcal{K}^{\left\langle M\right\rangle}$ is defined canonically by
\begin{equation}\label{eq:process-def}
Z_t^M = \pi_M(Z_t),
\end{equation}
where $\pi_M: \mathcal{K}^{\left\langle \infty  \right\rangle} \to \mathcal{K}^{\left\langle M\right\rangle}$ is the projection from \eqref{eq:piem}. 
Its symmetric transition probability density function $g_M(t,x,y): (0,\infty) \times \mathcal{K}^{\left\langle M \right\rangle} \times \mathcal{K}^{\left\langle M\right\rangle} \to (0,\infty)$ is given by
\begin{equation}
\label{eq:refldens}
g_{M}\left(t,x,y\right)= \left\{ \begin{array}{ll}
\displaystyle\sum_{y'\in \pi_{M}^{-1} (y)}{ g(t,x,y')} & \textrm{if } y \in \mathcal{K}^{\left\langle M\right\rangle} \backslash \cV_{M}^{\left\langle M\right\rangle}, \\
\displaystyle\sum_{y'\in \pi_{M}^{-1} (y)}{g(t,x,y')} \cdot \textrm{rank}_M(y') & \textrm{if } y \in \cV_{M}^{\left\langle M\right\rangle}, \\
\end{array}\right.
\end{equation}
where $\textrm{rank}_M(y')$ is the number of $M$-complexes meeting at the point $y'{\in \cV_M^{\langle \infty\rangle}}$. See \cite{bib:KOP} for details.

\subsubsection{Subordinate processes:\ the Brownian motion and the reflected Brownian motion}
Let $S=(S_t)_{t \geq 0}$ be a \emph{subordinator} on a probability space $(\Omega_0, \mathcal{F}, \mathcal{P})$, 
i.e., a nondecreasing L\'evy process taking values in $[0,\infty)$. 
The laws of $S_t$, denoted by $\eta_t(\mathrm{d}u) := \mathcal{P}(S_t \in \mathrm{d}u)$,  $t \ge 0$,
form a convolution semigroup of probability measures on $[0,\infty)$, uniquely determined by the Laplace transform
\begin{align}\label{eq:laplace}
\int_{[0,\infty)} {\rm e}^{-\lambda u}\, \eta_t(\mathrm{d}u) = {\rm e}^{-t \phi(\lambda)}, \quad \lambda>0.
\end{align}
The Laplace exponent $\phi$ is a \emph{Bernstein function} with $\phi(0+)=0$ which can be represented as
\begin{align} \label{eq:def_Phi}
\phi(\lambda) = b \lambda + \int_{(0,\infty)} \bigl(1 - {\rm e}^{-\lambda u}\bigr) \rho(\mathrm{d}u),
\end{align}
where $b \ge 0$ and $\rho$ is a nonnegative Radon measure on $(0,\infty)$ satisfying 
$\int_{(0,\infty)} (u \wedge 1)\, \rho(\mathrm{d}u) < \infty$. 
The constant $b$ and the measure $\rho$ are called the \emph{drift term} and the \emph{Lévy measure} of $S$, respectively. 
For further details, we refer to the monograph \cite{bib:SSV}.

\medskip

Throughout, we will assume the following.

	\begin{itemize}
\item[\bfseries(\namedlabel{B}{{\bf B}})] (\textbf{Bernstein function})
    \begin{enumerate}
    \item[a)] [Hartman-Wintner type condition] One has
		 \begin{equation}\label{eq:phi_at_infinity}
	\lim_{\lambda \to \infty} \frac{\phi(\lambda)}{\log \lambda} = \infty;
\end{equation}
		\medskip
    \item[b)]
   There exist $\alpha\in (0,d_w]$ and $C_1, C_2, \lambda_0 >0$ such that
		\begin{equation}\label{eq:phi_at_zero}
	C_1 \lambda^{\alpha/d_w}\leq 	 \phi(\lambda) \leq C_2 \lambda^{\alpha/d_w},\quad\lambda\in(0,\lambda_0].
	\end{equation}
    \end{enumerate}
\end{itemize}	

	\bigskip
\noindent	
Part a) of this condition is of technical nature, while b) is structural.

From \eqref{eq:phi_at_infinity} it follows that $\lim_{\lambda \to \infty} \phi(\lambda) = \infty$, and therefore either $b>0$ or $\int_{(0,\infty)} \rho({\rm d}u) = \infty$. In particular, from \eqref{eq:laplace} we have $\eta_t(\left\{0\right\})=0$, for every $t>0$.

    Typical examples of Bernstein functions (and the corresponding subordinators) that satisfy assumption \eqref{B} are:\\  
		$\phi(\lambda)=b \lambda$, $ b>0$ (\emph{pure drift}), \\
		$\phi(\lambda)=\lambda^{\alpha/d_w}$, $\alpha \in (0,d_w)$ (\emph{$\alpha/d_w$-stable subordinator}), \\
    $\phi(\lambda)=(\lambda+m^{d_w/\vartheta})^{\vartheta/d_w}-m$, $\vartheta \in (0,d_w)$, $m>0$ (\emph {relativistic $\vartheta/d_w$-stable subordinator}),
    and \\ $\phi(\lambda)=\sum_{i=1}^n \lambda^{\alpha_i/d_w}$, $\alpha_i \in (0,d_w]$, $n \in \left\{2,3,\ldots \right\}$ (\emph{mixture of stable subordinators}). \\ Further examples can be found e.g.\ in the monograph \cite{bib:SSV}.

\bigskip

Given a subordinator $S$ with Laplace exponent $\phi$ satisfying \eqref{B}, 
we define the \emph{subordinate Brownian motion} $X = (X_t)_{t \ge 0}$ 
and the \emph{subordinate reflected Brownian motion} 
$X^M = (X^M_t)_{t \ge 0}$ by
\[
X_t := Z_{S_t}, \qquad t \ge 0,
\]
and
\[
X^M_t := Z^M_{S_t}, \qquad t \ge 0,
\]
respectively.
For simplicity, we retain the notation $\mathbf P^x$, $x \in \mathcal K^{\langle\infty\rangle}$, and 
$\mathbf P_M^x$, $x \in \mathcal K^{\langle M\rangle}$, for the probability measures of these processes, 
as in the case of the diffusions $Z$ and $Z^M$ (here $x$ denotes the starting point). 
The corresponding expectations are denoted by $\ex^x$ and $\ex_M^x$, respectively.

By the general theory of subordination (see, e.g., \cite[Chapters 5 and 13]{bib:SSV}), 
the processes $X$ and $X^M$ are Feller processes with c\`adl\`ag paths.

Since $\eta_t(\{0\})=0$ for $t>0$, both $X$ and $X^M$ admit symmetric transition probability densities given by
\[
p(t,x,y)
=
\int_0^\infty g(u,x,y)\,\eta_t(\mathrm du),
\qquad
t>0,\;\; x,y \in \mathcal K^{\langle\infty\rangle},
\]
and
\begin{align}\label{eq:subord_pM}
p_M(t,x,y)
=
\int_0^\infty g_M(u,x,y)\,\eta_t(\mathrm du),
\qquad
t>0,\;\; x,y \in \mathcal K^{\langle M\rangle},
\end{align}
respectively. By exchanging the order of integration and summation, we further obtain
\begin{equation}
\label{eq:subrefldens}
p_{M}\left(t,x,y\right)= \left\{ \begin{array}{ll}
\displaystyle\sum_{y'\in \pi_{M}^{-1} (y)}{ p(t,x,y')} & \textrm{if } y \in \mathcal{K}^{\left\langle M\right\rangle} \backslash \cV_{M}^{\left\langle M\right\rangle}, \\
\displaystyle\sum_{y'\in \pi_{M}^{-1} (y)}{p(t,x,y')} \cdot \textrm{rank}_M(y') & \textrm{if } y \in \cV_{M}^{\left\langle M\right\rangle} \\
\end{array}\right.
\end{equation}
(recall that $\textrm{rank}_M(y')$ denotes the number of $M$-complexes meeting at the point $y'{\in \cV_M^{\langle \infty\rangle}}$).
Both densities $p(t,x,y)$ and $p_M(t,x,y)$ are bounded for every fixed $t>0$. 
Moreover, they inherit the continuity properties of the densities $g(t,x,y)$ and $g_M(t,x,y)$. These properties are proved in \cite[Lemma 2.1]{bib:HB-KK-MO-KPP}. Consequently, the processes $X$ and $X^M$ are strong Feller.

Several results in the subsequent sections rely on properties of bridge measures 
associated with the subordinate processes. 

For $t>0$ and $x,y \in \mathcal K^{\langle \infty \rangle}$, 
we denote by $\mathbf P^{x,y}_{t}$ a version of the conditional law of 
$(X_s)_{0 \le s \le t}$ under $\mathbf P^x$, given $X_t = y$. 
This is a probability measure on $D([0,t], \mathcal K^{\langle \infty \rangle})$ 
characterized by the disintegration formula
\begin{equation}\label{eq:bridge}
p(t,x,y)\mathbf{P}^{x,y}_{t}(A) 
= \mathbf{E}^x\!\left[\mathbf{1}_A \, p(t-s,X_s,y)\right],
\end{equation}
valid for every $0<s<t$ and every $A \in \sigma(X_u : u \le s)$.

For a precise construction and further properties of Markovian bridges 
for general Feller processes we refer to~\cite{bib:CU}.

Similarly, for the reflected processes $X^M$ we consider the bridge measures 
$\mathbf{P}^{x,y}_{M,t}$, defined for $M \in \mathbb Z_+$, $t>0$, 
and $x,y \in \mathcal K^{\langle M \rangle}$, 
as probability measures on $D([0,t], \mathcal K^{\langle M \rangle})$ 
satisfying
\begin{equation}\label{eq:Mbridge}
p_M(t,x,y)\mathbf{P}^{x,y}_{M,t}(A)
= \mathbf{E}^x\!\left[\mathbf{1}_A \, p_M(t-s,X^M_s,y)\right],
\end{equation}
for every $0 \le s < t$ and every 
$A \in \sigma(X^M_u : u \le s)$.

Expectations with respect to the measures 
$\mathbf{P}^{x,y}_{t}$ and $\mathbf{P}^{x,y}_{M,t}$ 
will be denoted by $\mathbf{E}^{x,y}_{t}$ and 
$\mathbf{E}^{x,y}_{M,t}$, respectively.

\subsection{Semigroups and generators of free processes} 
The transition semigroups of the processes under consideration form strongly 
continuous semigroups of self-adjoint contractions on the corresponding $L^2$-spaces. We now fix notation for 
these semigroups and their (negative definite) infinitesimal generators.

\begin{enumerate}
\item[(i)] For the Brownian motion $Z$ on $\mathcal K^{\langle\infty\rangle}$:\ the semigroup $(P_t)_{t\ge 0}$ on $L^2(\cK^{\langle \infty\rangle},\mathfrak{m})$ with generator $\mathcal L$.
\item[(ii)] For the reflected Brownian motion $Z^M$ on $\mathcal K^{\langle M \rangle}$, $M\in\mathbb Z$:\ the semigroup $(P_t^M)_{t\ge 0}$ on $L^2(\cK^{\langle M \rangle}, \mathfrak{m})$ with generator $\mathcal L^M$.
\item[(iii)] For the subordinate Brownian motion $X$ on $\mathcal K^{\langle\infty\rangle}$:\ the semigroup $(T_t)_{t\ge 0}$ on $L^2(\cK^{\langle \infty\rangle}, \mathfrak{m})$ with generator $-\phi(-\mathcal L)$.
\item[(iv)] For the reflected subordinate Brownian motion $X^M$ on $\mathcal K^{\langle M\rangle}$:\ the semigroup $(T_t^M)_{t\ge 0}$ on $L^2(\cK^{\langle M \rangle}, \mathfrak{m})$ with generator $-\phi(-\mathcal L^M)$.
\end{enumerate}
The operator $\mathcal L$ is the canonical Laplacian associated with the Brownian motion $Z$ on the fractal $\mathcal K^{\langle\infty\rangle}$, while $\mathcal L^M$ is the Neumann Laplacian defined on the complex $\mathcal K^{\langle M \rangle}$.
The notation in (iii) and (iv) indicates that the respective generators of the subordinate processes are in fact (minus) operator-monotone functions of the positive self-adjoint operators $-\mathcal L$ and $-\mathcal L^M$, which are defined via the spectral theorem.

We now briefly summarize the spectral properties of the operators in (ii) and (iv).
Since the semigroups $(P_t^M)_{t\ge 0}$ and $(T_t^M)_{t\ge 0}$ consist of Hilbert--Schmidt operators, 
the spectra of $-\mathcal L^M$ and $\phi(-\mathcal L^M)$ are purely discrete, and these operators admit complete sets of eigenfunctions. 

The eigenvalues of $-\mathcal L^M$ satisfy
\[
0 = \mu_1^M < \mu_2^M \le \mu_3^M \le \dots \rightarrow \infty,
\]
and they obey the scaling property
\begin{equation}\label{eq:ww_skalowanie}
\mu_k^M = L^{-Md_w} \, \mu_k^1, \qquad k = 1,2,3,\dots
\end{equation}

The eigenvalues of $\phi(-\mathcal L^M)$ are denoted by 
\[
\lambda_1^M < \lambda_2^M \le \lambda_3^M \le \dots \rightarrow \infty,
\]
and satisfy
\begin{equation}\label{eq:2.8}
\lambda_1^M = 0, \qquad 
\lambda_k^M = \phi(\mu_k^M) = \phi(L^{-Md_w} \, \mu_k^1), \quad k = 2,3,4,\dots
\end{equation}

The corresponding eigenfunctions form a complete orthonormal system in $L^2(\mathcal K^{\langle M\rangle}, \mathfrak{m})$, 
denoted by $(\psi_k^M)_{k\ge 1}$, with $\psi_1^M \equiv L^{-\frac{Md}{2}}$. 
Hence, we have
\begin{equation}\label{eq:2.9}
\phi(-\mathcal L^M)\psi_k^M = \lambda_k^M \psi_k^M, 
\qquad \text{and} \qquad
T_t^M \psi_k^M = {\rm e}^{-t \lambda_k^M} \psi_k^M.
\end{equation}

\subsection{Random Feynman--Kac semigroups and related Schr\"odinger operators} \label{subsec:FK}
In this paper, we consider the following class of random potentials. 

\smallskip

\begin{definition} \textbf{(Fractal Poisson-type potential)} \label{def:potendef} Let $W: \cK^{\langle \infty \rangle} \times \cK^{\langle \infty \rangle} \to [0,\infty)$ be a Borel function, and let $\mu^{\omega}$ be the random counting measure corresponding to a Poisson point process on $\cK^{\langle \infty \rangle}$ with intensity $\nu \textrm{d}\mathfrak{m}$, defined on a probability space $\left(\Omega, \cM, \mathbb{Q}\right)$, where $\nu >0$ is a given parameter. Then the random field
\begin{equation}
\label{eq:potendef}
V^{\omega}(x) := \int_{\cK^{\langle \infty \rangle}} W(x,y) \mu^{\omega}(\textrm{d}y),
\end{equation}
will be called a \emph{Poisson-type potential}. The two-argument profile function $W$ will be referred to as the \emph{single-site potential}.
\end{definition}

We assume that the Poisson point process and the subordinate Brownian motion are independent. 
Our techniques rely crucially on the periodic structure of the state space described in Section~\ref{sec:glp}.
In particular, our arguments involve the following \emph{periodized fractal Poisson-type potential}, defined by
\begin{equation} \label{eq:potendef-per}
V_{M}^{\omega}(x)
:=
\int_{\mathcal K^{\langle M\rangle}}
\sum_{y' \in \pi_M^{-1}(y)}
W(x,y') \, \mu^{\omega}(\mathrm dy),
\qquad M \in \mathbb Z_+ .
\end{equation}
This periodization can be described as follows:\ we first restrict the Poisson configuration to $\mathcal K^{\langle M\rangle}$, then extend it periodically to $\mathcal K^{\langle\infty\rangle}$, and finally attach the profile $W$ at all resulting points.

To proceed, we need that the potentials we consider belong to the Kato-classes of respective processes.
\begin{definition} \textbf{(Kato class and local Kato class)} \label{def:kato}
\begin{enumerate}
\item[(i)] We say that a measurable function 
$f:\mathcal K^{\langle\infty\rangle}\to\mathbb R$ 
belongs to the Kato class of the process $X$, denoted by $\mathcal J^X$, if
\begin{equation}\label{eq:kato-1}
\lim_{t\to 0}\sup_{x\in \mathcal K^{\langle\infty\rangle}}
\int_0^t T_s |f|(x)\,\mathrm ds = 0.
\end{equation}
The Kato class of the process $X^M$, denoted by $\mathcal J^{X^M}$, 
is defined analogously, with $\mathcal K^{\langle\infty\rangle}$ replaced by 
$\mathcal K^{\langle M\rangle}$ and $T_t$ replaced by $T_t^M$.

\item[(ii)] A measurable function 
$f:\mathcal K^{\langle\infty\rangle}\to\mathbb R$ 
belongs to the local Kato class of the process $X$, denoted by $\mathcal J_{\mathrm{loc}}^X$, 
if for every $M$-complex $\Delta$, the function 
$f\,\mathbf 1_{\Delta}$ belongs to $\mathcal J^X$.
\end{enumerate}
\end{definition}
\noindent
Since $\mathcal K^{\langle M\rangle}$ is compact, the local Kato class for the process $X^M$ need not be introduced.

\bigskip

We now impose general regularity assumptions on the single-site potential $W$ under which the random potentials considered are well defined and the corresponding Schr\"odinger operators are sufficiently regular. In particular, assumption \eqref{W1} ensures the existence of the IDS in our present setting.

\medskip

	\begin{itemize}
\item[\bfseries(\namedlabel{W1}{{\bf W1}})] (\textbf{Single-site potential})
    \begin{enumerate}
    \item[a)]
    $W(\cdot,y) \in \cJ^{X}_{\textrm{loc}}$ for every $y \in \cK^{\langle \infty \rangle}$ (cf.\ Definition \ref{def:kato}) and there exists a function $h \in L^{1}(\cK^{\langle \infty \rangle}, \mathfrak{m})$ such that for any $M\in\mathbb Z,$ if $x \in \cK^{\langle M \rangle}$ and $y \notin \cK^{\langle M+1 \rangle}$ then $W(x,y) \leq h(y).$
    \medskip
   \item[b)]
   $\sum_{M=1}^{\infty} \sup_{x \in \cK^{\langle \infty \rangle}}  \int_{\cK^{\langle \infty \rangle} \setminus \cC_{\lfloor M/4 \rfloor}(x)} W(x,y){\rm d}\mathfrak{m}(y) < \infty   $;
	\medskip
    \item[c)]
		For sufficiently large $M \in \mathbb{Z}_{+}$,
\begin{equation}
\label{eq:a3cond}
\sum_{y' \in \pi_M^{-1}(\pi_M(y))} W\left(\pi_M(x),y'\right) \leq \sum_{y' \in \pi_M^{-1}(\pi_M(v))} W\left(\pi_{M+1}(x),y'\right), \quad x, y \in \cK^{\langle \infty \rangle}. 
\end{equation}
    \end{enumerate}
\end{itemize}

\medskip

The profile function $W(x,y)$ describes the strength of the influence of a single Poisson site located at $y$ on a particle at the point $x$. 
Conditions (\ref{W1}.ab) are regularity assumptions, whereas condition (\ref{W1}.c) is of structural importance. 
It states that for every fixed $x,y \in \mathcal K^{\langle \infty \rangle}$, the influence of the sites in 
$\pi_M^{-1}(\pi_M(y))$ on a particle located at $\pi_M(x) \in \mathcal K^{\langle M \rangle}$ 
is, on average, smaller than the influence at the position $\pi_{M+1}(x)$.

The following proposition shows that under condition~\eqref{W1}, the potentials $V^\omega$ and $V_M^\omega$, defined in 
\eqref{eq:potendef} and \eqref{eq:potendef-per}, are well defined and belong to the local Kato class for almost every $\omega \in \Omega$.

\begin{proposition}
\label{pro:katoclass}
Let $V^\omega$ be the Poissonian potential \eqref{def:potendef} whose profile $W$ satisfies \eqref{W1}, and let $V^\omega_M$ be its periodization \eqref{eq:potendef-per}, $M\in\mathbb Z.$ Then, $\mathbb Q-$almost surely:\
\begin{itemize}
\item[(a)] $V^\omega \in \mathbb \cJ_{loc}^X$;
\item[(b)] $V^\omega_M\in \mathbb \cJ_{loc}^X$;
\item[(c)] $V^\omega_M\in \mathbb \cJ^{X^M}.$
\end{itemize}
\end{proposition}

\begin{proof} 
\noindent 
(a) First observe that
\[
\mathbb{E}^\mathbb{Q} \int_{\mathcal K^{\langle \infty\rangle}} 
h(y)\,\mu^\omega({\rm d}y)
= \nu \int_{\mathcal K^{\langle \infty\rangle}} 
h(y)\, \mathfrak{m}({\rm d}y)
= \nu \|h\|_{L^1}
< \infty.
\]
Consequently, there exists a set $\Omega' \subset \Omega$ of full probability such that 
for every $\omega \in \Omega'$ we have $h \in L^1(\mathcal K^{\langle \infty\rangle}, \mu^\omega)$.
For $n \in \mathbb Z_+$ define
\[
\Omega_n 
= \left\{ \omega \in \Omega : 
\text{only finitely many points fall into } 
\mathcal K^{\langle n\rangle} \right\}.
\]
Then $\mathbb Q(\Omega_n) = 1$ for every $n \in \mathbb Z_+$, 
and hence the set 
\[
\widetilde{\Omega} := \Omega' \cap \bigcap_{n \in \mathbb Z_+} \Omega_n
\]
also has full measure. 
For $\omega \in \widetilde{\Omega}$, we denote by $\{y_i(\omega)\}$ the realization of the Poisson cloud.

\smallskip

Let $\omega\in \widetilde{\Omega}.$ It is enough to show that condition \eqref{eq:kato-1} holds for $V^\omega\mathbf 1_{\mathcal K^{\langle n\rangle}},$ for every $n\in\Z_{+}.$
For $x\in \cK^{\langle \infty\rangle}$ we have
\[V^\omega(x)\mathbf 1_{\cK^{\langle n\rangle}}(x) \leq
\mathbf 1_{\cK^{\langle n\rangle}}(x)
\left(\sum_{y_i(\omega)\in \cK^{\langle n+1\rangle}} W(x, y_i(\omega)) +
\int_{(\cK^{\langle n+1\rangle})^c} h(y)\mu^\omega({\rm d}y)\right)\]
and further
\begin{align*}
\int_0^t & T_s(V^\omega\mathbf 1_{\cK^{\langle n\rangle}})(x)\,{\rm d}s = \int_0^t \mathbf E^x \left(V^\omega(X_s)\mathbf 1_{\cK^{\langle n\rangle}}(X_s)\right){\rm d}s\\
&\leq \int_0^t \mathbf E^x\left(\mathbf 1_{\cK^{\langle n\rangle}}(X_s)
\left(\sum_{y_i(\omega)\in \cK^{\langle n+1\rangle}} W(X_s, y_i(\omega)) +
\int_{(\cK^{\langle n+1\rangle})^c} h(y)\mu^\omega({\rm d}y)\right)\right)\\
&= \sum_{y_i(\omega)\in \cK^{\langle n+1\rangle}} \int_0^t T_s(W(\cdot,y_i(\omega))\mathbf 1_{\cK^{\langle n\rangle}})(x){\rm d}s
+ t\int_{(\cK^{\langle n+1\rangle})^c} h(y)\mu^\omega({\rm d}y).
\end{align*}
As we have assumed that $W(\cdot,y)\in\cJ_{loc}^X,$ each of the integrals in the first sum goes to zero as $t\to 0,$ uniformly over $x.$ Since there are finitely many terms, entire sum goes to 0 uniformly over $x$ as well. The second term does not depend on $x,$ and can be estimated by $t\|h\|_{L^1(\cK^{\langle \infty\rangle},\mu^\omega)}.$ The assertion follows.

\smallskip
\noindent
(b) The argument is essentially the same; the only additional step is to analyze 
the configuration periodized with respect to $\pi_M$. 
Denote by
\[
\{y_i^M(\omega)\}
:= \pi_M^{-1}\!\left(\{y_i(\omega)\} \cap \mathcal K^{\langle M\rangle}\right)
\]
the realization of the periodized cloud and set 
$\mu_M^\omega := \mu^\omega \circ \pi_M$. Then
\begin{align*}
\mathbb{E}^\mathbb{Q} \int_{\mathcal K^{\langle \infty\rangle}} 
h(y)\,\mu_M^\omega({\rm d}y)
&= \mathbb{E}^\mathbb{Q} \int_{\mathcal K^{\langle M\rangle}} 
\sum_{y' \in \pi_M^{-1}(y)} h(y')\,\mu^\omega({\rm d}y) \\
&= \nu \int_{\mathcal K^{\langle M\rangle}} 
\sum_{y' \in \pi_M^{-1}(y)} h(y')\,\mathfrak{m}({\rm d}y)
= \nu \int_{\mathcal K^{\langle \infty\rangle}} 
h(y)\, \mathfrak{m}({\rm d}y)
= \nu \|h\|_{L^1}
< \infty.
\end{align*}
Hence we may also assume that $h \in L^1(\mathcal K^{\langle M\rangle}, \mu_M^\omega)$ for $\omega \in \widetilde{\Omega}$.

Now let $\omega \in \widetilde{\Omega}$, $n \ge M$, and 
$x \in \mathcal K^{\langle \infty\rangle}$. Then
\begin{align*}
\mathbf 1_{\mathcal K^{\langle n\rangle}}(x) V_M^\omega(x)
&= \mathbf 1_{\mathcal K^{\langle n\rangle}}(x)
\int_{\mathcal K^{\langle M\rangle}}
\sum_{y' \in \pi_M^{-1}(y)} W(x,y')\,\mu^\omega({\rm d}y) \\
&= \mathbf 1_{\mathcal K^{\langle n\rangle}}(x)
\int_{\mathcal K^{\langle \infty\rangle}}
W(x,y)\,\mu_M^\omega({\rm d}y) \\
&\leq
\mathbf 1_{\mathcal K^{\langle n\rangle}}(x)
\left(
\sum_{y_i^M(\omega) \in \mathcal K^{\langle n+1\rangle}}
W(x,y_i^M(\omega))
+
\int_{(\mathcal K^{\langle n+1\rangle})^c}
h(y)\,\mu_M^\omega({\rm d}y)
\right).
\end{align*}
The conclusion now follows exactly as in part~(a).

\smallskip
\noindent
(c) By \cite[Theorem 3.1]{bib:O} and the subordination formula \eqref{eq:subord_pM}, there are constants $c_1, c_2>0$ such that
\[
p_M(s,x,y) \leq c_1\big(p(c_2s,x,y) + L^{-Md}\big), \quad x ,y \in \cK^{\langle M\rangle}, \ s>0, \ M \in \Z_{+}. 
\]
Hence, 
\begin{eqnarray*}
\int_0^t T_s^M V^\omega_M(x){\rm d}s \leq c_1 \int_0^tT_s V^\omega_M(x){\rm d}s + c_1 t L^{-Md} \int_{\cK^{\langle M\rangle}}V_M^\omega(y)\mathfrak{m}({\rm d}y).
\end{eqnarray*}
The first summand on the right hand side goes to zero uniformly over $x\in\cK^{\langle M\rangle}$ due to (b), and the second - because $V_M^\omega\in \cJ_{loc}^X\subset L^1_{loc}(\cK^{\langle\infty\rangle},\mathfrak{m}),$ almost surely. The proof of the proposition is complete.

\end{proof}

Given a fractal  Poisson-type potential $V^{\omega}$ which belongs almost surely to the local Kato class $\mathcal J_{loc}^X$, 
we define the random Feynman--Kac semigroup $\{T^{V^\omega}_t : t \ge 0\}$ of the subordinate Brownian motion $X$ by
\[
T^{V^\omega}_t f(x) =  \mathbf{E}^x \left[{\rm e}^{-\int_0^t V^{\omega}(X_s) {\rm d}s}f(X_t) \right], \quad f \in L^2(\cK^{\langle \infty\rangle}, \mathfrak{m}), \ \ t>0.
\]
This semigroup is strongly continuous on $L^2(\mathcal K^{\langle \infty\rangle}, \mathfrak{m})$ and consists of self-adjoint operators.
It is well known that $T^{V^\omega}_t = {\rm e}^{-tH^{\omega}}$, where 
\[
H^{\omega} = \phi(-\mathcal L) + V^{\omega}
\]
is the random Schr\"odinger operator associated with the generator of the subordinate Brownian motion $X$ on $\mathcal K^{\langle \infty\rangle}$. 
Here $\phi$ denotes the Laplace exponent of the subordinator $S$, and $\mathcal L$ is the generator of the free Brownian motion on the USNF. 
A standard reference on Schr\"odinger operators generated by Feller processes and their Feynman--Kac semigroups is the monograph~\cite{bib:DC} by Demuth and van Casteren (see also \cite[Chapters~3.2--3.3]{bib:CZ}).

We also define random Feynman--Kac semigroups 
$\{T_t^{D,M,V^\omega} : t\ge 0\}$ and 
$\{T_t^{N,M,V^\omega} : t\ge 0\}$ 
corresponding, respectively, to the process $X$ killed upon exiting the complex $\mathcal K^{\langle M\rangle}$ 
and to the reflected process $X^M$ evolving in $\mathcal K^{\langle M\rangle}$. 
More precisely,
\[
T_t^{D,M,V^\omega} f(x) = \mathbf{E}^x \left[{\rm e}^{-\int_0^t V^{\omega}(X_s) {\rm d}s}f(X_t); t<\tau_{\cK^{\langle M\rangle}} \right], \quad f \in L^2(\cK^{\langle M\rangle}, \mathfrak{m}), \ M \in \mathbb{Z}_+, \ t>0,
\]
and
\[
T_t^{N,M,V^\omega} f(x) = \mathbf{E}^x_M \left[{\rm e}^{-\int_0^t V^{\omega}(X^M_s) {\rm d}s}f(X^M_t) \right], \quad f \in L^2(\cK^{\langle M\rangle}, \mathfrak{m}), \ M \in \mathbb{Z}_+, \ t>0.
\]
Here $\tau_{\mathcal K^{\langle M\rangle}} = \inf \{ t \ge 0 : X_t \notin \mathcal K^{\langle M\rangle} \}$ denotes the first exit time of the process from $\mathcal K^{\langle M\rangle}$.

Denote by $A^{D,M,V^\omega}$ and $A^{N,M,V^\omega}$ the $L^2$-generators of the semigroups 
$\{T_t^{D,M,V^\omega}:t \geq 0\}$ and $\{T_t^{N,M,V^\omega}:t \geq 0\}$, respectively. 
As mentioned above, killing and reflecting the process correspond to imposing 
\emph{Dirichlet} and \emph{Neumann boundary conditions} on the generator, respectively. 
Therefore, the (positive definite) finite-volume operators
\begin{equation}\label{eq:generators}
H_{M}^{D,V^\omega} := - A^{D,M,V^\omega}, 
\qquad 
H_{M}^{N,V^\omega} := - A^{N,M,V^\omega},
\end{equation}
may be viewed as Schr\"odinger operators associated with the Dirichlet 
(resp.\ Neumann) realizations of the generator of the initial subordinate Brownian motion $X$. 
In particular, by \cite[Theorem~2.5]{bib:DC},
\[
H_{M}^{N,V^\omega} = \phi(-\mathcal L^M) + V^\omega.
\]

The operators $T_t^{D,M,V^\omega}$ and $T_t^{N,M,V^\omega}$, $t>0$, are integral operators. 
For every $t>0$ there exist symmetric, bounded kernels 
$u_{M,\omega}^D(t,x,y)$ and $u_{M,\omega}^N(t,x,y)$ such that
\begin{equation*}
T_t^{D,M,V^\omega} f(x) 
= \int_{\mathcal K^{\langle M \rangle}} 
u_{M,\omega}^D(t,x,y)\, f(y)\, \mathfrak{m}({\rm d}y),
\qquad f \in L^2(\mathcal K^{\langle M\rangle}, \mathfrak{m}),
\end{equation*}
and
\begin{equation*}
T_t^{N,M,V^\omega} f(x) 
= \int_{\mathcal K^{\langle M \rangle}} 
u_{M,\omega}^N(t,x,y)\, f(y)\, \mathfrak{m}({\rm d}y),
\qquad f \in L^2(\mathcal K^{\langle M\rangle}, \mathfrak{m}).
\end{equation*}
These kernels admit the following bridge representations:
\begin{equation*}
u_{M,\omega}^D(t,x,y) 
= p(t,x,y)\,
\mathbf{E}_{t}^{x,y} \!\left[
\exp\!\left(-\int_0^t V^{\omega}(X_s)\,{\rm d}s\right);
\, t<\tau_{\mathcal K^{\langle M\rangle}}
\right],
\end{equation*}
and
\begin{equation*}
u_{M,\omega}^N(t,x,y) 
= p_M(t,x,y)\,
\mathbf{E}_{M,t}^{x,y} \!\left[
\exp\!\left(-\int_0^t V^{\omega}(X^M_s)\,{\rm d}s\right)
\right].
\end{equation*}

Since for every $t>0$ these kernels are bounded and 
$\mathfrak{m}(\mathcal K^{\langle M\rangle}) < \infty$, 
the operators $T_t^{D,M,V^\omega}$ and $T_t^{N,M,V^\omega}$ are Hilbert--Schmidt. 
Consequently, the spectra of $H_{M}^{D,V^\omega}$ and $H_{M}^{N,V^\omega}$ 
consist of isolated eigenvalues of finite multiplicity, satisfying
\[
0 \le \lambda_1^{D,M,V^\omega} 
< \lambda_2^{D,M,V^\omega} 
\le \lambda_3^{D,M,V^\omega} 
\le \ldots \to \infty,
\]
and
\[
0 \le \lambda_1^{N,M,V^\omega} 
< \lambda_2^{N,M,V^\omega} 
\le \lambda_3^{N,M,V^\omega} 
\le \ldots \to \infty,
\]
respectively.

\subsection{Finite volume spectral measures, their Laplace transforms and existence of IDS} \label{sec:IDS_ex}
For a large class of fractals with well-defined periodic structures, 
namely USNF's with the GLP, one can establish the existence of the 
Integrated Density of States for the Schr\"odinger operators $H^{\omega}$. 
As usual, we consider the random empirical measures on $[0,\infty)$, built on the spectra of $H_{M}^{D,V^\omega}$ and $H_{M}^{N,V^\omega}$, normalized by $\mathfrak{m}(\cK^{\langle M\rangle})$:
\begin{equation}
\label{eq:lmd}
\Lambda_M^{D, V^\omega} := \frac{1}{\mathfrak{m}\left(\cK^{\langle M\rangle}\right)} \sum_{n=1}^{\infty} \delta_{\lambda_n^{D,M,V^\omega}}
\end{equation}
and
\begin{equation}
\label{eq:lmn}
\Lambda_M^{N, V^\omega} := \frac{1}{\mathfrak{m}\left(\cK^{\langle M\rangle}\right)} \sum_{n=1}^{\infty} \delta_{\lambda_n^{N,M,V^\omega}} \ .
\end{equation}
We show that these measures converge to a common limit $\Lambda$, which is a nonrandom measure on $[0,\infty)$. This limit is referred to as the \emph{Integrated Density of States} (IDS).

\begin{theorem} \label{thm:IDS}
Let $\cK^{\langle \infty \rangle}$ be an USNF with the GLP and let the assumptions \eqref{B} and \eqref{W1} hold.
Then the random measures $\Lambda_M^{D, V^\omega}$ and $\Lambda_M^{N, V^\omega}$ are $\mathbb{Q}$-almost surely
vaguely convergent as $M \to \infty$ to a common nonrandom limit measure $\Lambda$ on $[0,\infty)$.
\end{theorem}

\noindent
The proof of Theorem~\ref{thm:IDS} follows the approach of \cite{bib:KaPP2}, where the existence of the IDS was established for random Schr\"odinger operators with Poisson potentials associated with subordinate Brownian motions on the Sierpi\'nski triangle, the most regular planar simple nested fractal. A closely related scheme was applied more recently to the same class of processes 
and state spaces, but for a different type of random potential, namely the fractal alloy-type potentials induced by the lattice 
$\mathcal V_{0}^{\langle 0\rangle}$~\cite{bib:HB-KK-MO-KPP}. For completeness, a concise account is provided in Appendix~\ref{sec:appA}. Here we only sketch the main idea which will also be essential in the analysis of the Lifshitz singularity forming a central theme of the present work.

In general, the argument relies on approximating the operator $H^{\omega}$, acting in 
$L^2(\mathcal K^{\langle \infty\rangle}, \mathfrak{m})$, by the operators defined on 
$L^2(\mathcal K^{\langle M \rangle}, \mathfrak{m})$ as $M \nearrow \infty$. 
The key ingredient is the sequence of reflected processes on $\mathcal K^{\langle M \rangle}$, $M \in \mathbb Z$, 
together with the folding projections $\pi_M$ constructed in \cite{bib:KOP}. 
Thanks to assumption (\ref{W1}.c), suitable manipulations of the potential 
recover the necessary monotonicity properties, which are essential for the proof.

The convergence of the measures $\Lambda_M^{D, V^\omega}$ and $\Lambda_M^{N, V^\omega}$ is deduced from the convergence of their Laplace transforms $
\mathbb L_M^{D,V^\omega},
\mathbb L_M^{N,V^\omega}.$
Observe that
\begin{align*}
\mathbb L_M^{D,V^\omega} (t) & = \int_0^{\infty} {\rm e}^{-\lambda t} \Lambda_M^{D, V^\omega} ({\rm d}\lambda) \\
& = \frac{1}{\mathfrak{m}\left(\cK^{\langle M\rangle}\right)} \sum_{n=1}^{\infty} {\rm e}^{-\lambda_n^{D,M,V^\omega}t} =  \frac{1}{\mathfrak{m}\left(\cK^{\langle M\rangle}\right)}	\textrm{Tr} T_t^{D,M,V^\omega} \\
& =  \frac{1}{\mathfrak{m}\left(\cK^{\langle M\rangle}\right)} \int_{\cK^{\langle M\rangle}} u_{M,\omega}^D(t,x,x) \, \mathfrak{m}({\rm d}x) \\
& = \frac{1}{\mathfrak{m}\left(\cK^{\langle M\rangle}\right)} \int_{\cK^{\langle M\rangle}} p(t,x,x) \mathbf{E}_{t}^{x,x} \left[{\rm e}^{-\int_0^t V^{\omega}(X_s){\rm d}s}; t<\tau_{\cK^{\langle M\rangle}} \right] \mathfrak{m}({\rm d}x)
\end{align*}
and
\begin{align*}
\mathbb L_M^{N,V^\omega} (t) &  = \int_0^{\infty} {\rm e}^{-\lambda t} \Lambda_M^{N, V^\omega} ({\rm d}\lambda) \\
& = \frac{1}{\mathfrak{m}\left(\cK^{\langle M\rangle}\right)} \sum_{n=1}^{\infty} {\rm e}^{-\lambda_n^{N,M,V^\omega}t} =  \frac{1}{m\left(\cK^{\langle M\rangle}\right)}	\textrm{Tr} T_t^{N,M,V^\omega} \\
& =  \frac{1}{\mathfrak{m}\left(\cK^{\langle M\rangle}\right)} \int_{\cK^{\langle M\rangle}} u_{M,\omega}^N(t,x,x) \, \mathfrak{m}({\rm d}x) \\
 & =  \frac{1}{\mathfrak{m}\left(\cK^{\langle M\rangle}\right)} \int_{\cK^{\langle M\rangle}} p_M(t,x,x) \mathbf{E}_{M,t}^{x,x} \left[{\rm e}^{-\int_0^t V^{\omega}(X^M_s){\rm d}s}\right] \mathfrak{m}({\rm d}x) .
\end{align*}
In fact, we do not approach the convergence of $\mathbb L_M^{D,V^\omega}$ and $\mathbb L_M^{N,V^\omega}$ directly. The line of attack is as follows.
For a given $M$, we first replace the initial Poisson potential $V^{\omega}$ with the potential $V_{M}^{\omega}$, constructed from the periodized configuration of Poisson points via the mapping $\pi_M$, see \eqref{eq:potendef-per}.
Then, we consider the Feynman--Kac semigroup consisting of operators
\begin{equation*}
T_t^{N,M,V^\omega_M} f(x) = \mathbf{E}^x_M \left[{\rm e}^{-\int_0^t V_{M}^{\omega}(X^M_s) {\rm d}s}f(X^M_t) \right], \quad f \in L^2(\cK^{\langle M\rangle}, \mathfrak{m}), \ M \in \mathbb{Z}_+, \ t>0,
\end{equation*}
and the corresponding Neumann Schr\"odinger operator
$$
H_{M}^{N,V^\omega_M} =  \phi(-\cL^M)  + V_{M}^{\omega}.
$$
Finally, we prove the convergence of the Laplace transform $\mathbb L_M^{N,V^\omega_M}(t)$ to a finite limit $\mathbb L(t)$ (identified with the Laplace transform of the IDS). More precisely, in Theorem \ref{thm:convergence_with_star} we show that for every $t>0$,
\begin{align}\label{eq:monot_conv_Laplace}
\mathbb{E}^{\mathbb{Q}}\mathbb L_M^{N,V^\omega_M}(t) \searrow \mathbb L(t) \quad  \text{as} \quad M \to \infty.
\end{align}
Furthermore, since
\begin{align*}
\mathbb L_M^{N,V^\omega_M}(t) = \int_0^{\infty} {\rm e}^{-\lambda t} \Lambda_{M}^{N,V^\omega_M} ({\rm d}\lambda)
& = \frac{1}{\mathfrak{m}(\mathcal K^{\langle M\rangle})} \sum_{n=1}^\infty {\rm e}^{-t \lambda_n^{N,M,V^\omega_M}}
\\
&\leq   \frac{1}{\mathfrak{m}(\mathcal K^{\langle M\rangle})} \left[
{\rm e}^{-(t-1)\lambda_1^{N,M,V^\omega_M}} \mbox{Tr} T_1^{N,M,V^\omega_M}\right] \\
&\leq
{\rm e}^{-(t-1)\lambda_1^{N,M,V^\omega_M}}\frac{1}{\mathfrak{m}(\cK^{\langle M\rangle})} \int_{\mathcal K^{\langle M\rangle}}p_M(1,x,x)\mathfrak{m}({\rm d}x),
\end{align*}
for $t>1$, the monotone convergence in \eqref{eq:monot_conv_Laplace} and \cite[Lemma 2.1(b)]{bib:HB-KK-MO-KPP} give
\begin{equation}\label{eq:trace-estimate}
 \mathbb L(t) \leq C \mathbb{E}^{\mathbb{Q}} {\rm e}^{-(t-1)\lambda_1^{N, M,V^\omega_M}}, \quad t>1, \ M \in \mathbb{Z}_{+},
\end{equation}
with a uniform constant $C>0$, not depending on $M.$ This estimate will be a starting point in the proof of the upper bound for $\mathbb L(t)$. Another important consequence of \eqref{eq:monot_conv_Laplace} is that we also have the convergence (not necessarily monotone)
\begin{align}\label{eq:dirichlet_conv_Laplace}
\mathbb{E}^{\mathbb{Q}}\mathbb L_{M}^{D,V^{\omega}}(t) \rightarrow \mathbb L(t) \quad  \text{as} \quad M \to \infty,
\end{align}
see Corollary \ref{cor:final}. It will be an initial step in the proof of the lower bound for $\mathbb L(t)$.

\section{Upper bound for the Laplace transform of the IDS} \label{sec:LS_general}

Our main Theorem \ref{th:main_IDS}
will be proven under the following assumption \eqref{W2}:

\bigskip

	\begin{itemize}
\item[\bfseries(\namedlabel{W2}{{\bf W2}})] (\textbf{Single-site potential with regular and bounded support})
    \begin{enumerate}
		\item[a)]
		There exist $m_0\in\mathbb Z $ and $A_0>0$ such that
\begin{equation}\label{eq:assump-W}
W(x,y) \geq A_{0} \quad  \text{for all}  \quad  x,y\in \mathcal K^{\langle\infty\rangle} \mbox{ with } d_{m_0}(x,y) \leq 1.
\end{equation}
    \item[b)]
		There exists $M_0 \in \mathbb{Z}_{+}$ such that
		$$W(x,y)=0 \quad \text{for all} \quad x,y \in \mathcal K^{\langle\infty\rangle}   \text{ with }   d_{M_0}(x,y) > 1;$$
    \end{enumerate}
\end{itemize}
\bigskip
\noindent
Condition (\ref{W2}.a) means that the profile $W(x,y)$ is (uniformly) separated from 0 when $x$ and $y$ are close, and condition (\ref{W2}.b) -- that $W(x,y)$ vanishes when $x$ and $y$ are far away.

In this section, we find a general upper estimate for the Laplace transform $\mathbb L(t)$ of the IDS for single-site potentials $W$ with regular support, i.e.\ we assume that (\ref{W2}.a) holds. It does not require any assumption on the boundedness of the support of $W$ (in particular, (\ref{W2}.b) is not needed here). However, in Section \ref{sec:LS_bounded} this result will be complemented with a matching lower bound  on  $\mathbb L(t)$ for single-site potentials $W$ with bounded  support. 

The following theorem is our main result in this section.

\begin{theorem}\label{th:main} Assume that the characteristic exponent $\phi$ satisfies \textup{(\ref{B})}, and that the single-site potential $W$ satisfies \textup{(\ref{W2}.a)} and \textup{(\ref{W1})}. Then for every $\nu_0>0$ there exists a constant $C = C(\nu_0)>0$ such that for every $\nu \geq \nu_0$
\begin{equation}\label{eq:upper-poiss-1}
\limsup_{t\to\infty} \frac{\log \mathbb L(t)}{t^{\frac{d}{d+\alpha}}}\leq -C\nu^{\frac{\alpha}{d+\alpha}}.
\end{equation}
\end{theorem}

\noindent
In light of \eqref{eq:trace-estimate}, the proof of the above theorem can be reduced to finding a satisfactory upper estimate for the expectation $$\mathbb{E}^{\mathbb{Q}} {\rm e}^{-(t-1)\lambda_1^{N, M,V^{\omega}_M}}, $$ where $\lambda_1^{N, M,V^{\omega}_M}$ is the the ground state eigenvalue of the random Schr\"odinger operator $H_{M}^{N,V^{\omega}_M}$. Our reasoning consists of three main steps.

\subsection*{STEP 1. Reduction to an alloy-type potential indexed by complexes of the fractal}
 This step is critical for the entire proof. The key observation is that we can replace the potential
$V_{M}^{\omega}$ by an alloy-type random potential in which the sites are modeled by complexes of the fractal of order $m_0$.
Formally, we consider $m_0 \in \mathbb Z$ from assumption (\ref{W2}.a) and introduce the collection of random variables $\big\{q_{\Delta}: \Delta \in \cT_{m_0}\big\}$, which indicate whether the set $\Delta \setminus \mathcal V^{\langle \infty\rangle}_{m_0}$ received any of Poisson points or not, i.e.
$$
q_{\Delta}(\omega):= \mathbf{1}_{\big\{\Delta \setminus \mathcal V^{\langle \infty\rangle}_{m_0} \ \text{received any of Poisson points} \big\}}(\omega), \quad \Delta \in \cT_{m_0}
$$
(recall that $\cT_{m_0}$ denotes the family of all $m_0$-complexes in $\cK^{\langle \infty \rangle}$).
It then follows directly from the definition that $q_{\Delta}$ are i.i.d.\ random variables on $\left(\Omega, \cM, \mathbb{Q}\right)$ with Bernoulli distribution
$$
\mathbb Q[q_{\Delta}=0]={\rm e}^{-\nu L^{dm_0}} \quad \text{and} \quad \mathbb Q[q_{\Delta}=1]= 1-{\rm e}^{-\nu L^{dm_0}}.
$$
With this in mind, we set
\begin{align}
{\overline{V}}_{M}^{\, \omega}(x) & = A_0 \cdot \sum_{ \Delta \in \cT_{m_0} \atop \Delta \subset \cK^{\langle M \rangle}} q_{\Delta}(\omega) \cdot \mathbf 1_{\Delta \setminus \mathcal V^{\langle \infty\rangle}_{m_0}}(\pi_M(x)) \label{eq:period_alloy}\\
& = A_0 \cdot q_{\Delta(\pi_M(x))}(\omega) \cdot \mathbf 1_{\cK^{\langle \infty \rangle} \setminus \mathcal V^{\langle \infty\rangle}_{m_0}}(x) ,\qquad x\in \cK^{\langle \infty \rangle}, \ M\in \mathbb Z_+, \ M\geq m_0, \nonumber
\end{align}
where the constant $A_0>0$ comes from the assumption (\ref{W2}.a) (if $x \in \cK^{\langle \infty \rangle} \setminus \mathcal V^{\langle \infty\rangle}_{m_0}$, then $\Delta_{m_0}(x)$ denotes the unique $m_0$-complex containing $x$). By (\ref{W2}.a),
\begin{align}\label{eq:poten_monot}
V_{M}^{\omega}(x) \geq {\overline{V}}_{M}^{\, \omega}(x), \quad x \in \cK^{\langle \infty \rangle}.
\end{align}
Indeed, since $\mu^{\omega}(B)$ counts the Poisson points that fell onto $B$, we get
\begin{align*}
V_{M}^{\omega}(x) & = \int_{\mathcal K^{\langle M\rangle}} \sum_{y{'} \in \pi_M^{-1}(y)} W(x,y{'}) \mu^{\omega}({\rm d}y) \\
                  & \geq \mathbf 1_{\cK^{\langle \infty \rangle} \setminus \mathcal V^{\langle \infty\rangle}_{m_0}}(x)
									   \cdot A_0 \cdot \mu^{\omega}\big(\Delta_{m_0}(\pi_M(x))\big) \\
                  & \geq \mathbf 1_{\cK^{\langle \infty \rangle} \setminus \mathcal V^{\langle \infty\rangle}_{m_0}}(x) \cdot A_0 \cdot q_{\Delta_{m_0}(\pi_M(x))}(\omega)
										= {\overline{V}}_{M}^{\, \omega}(x).
\end{align*}
\begin{remark}[\textbf{General alloy-type potentials indexed by fractal complexes}]\label{rem:complex_alloy}{\rm
It is instructive to notice that the potential ${\overline{V}}_{M}^{\, \omega}(x)$ introduced above is a special case of a general alloy-type potential in which the sites are no longer given by the lattice vertices, but they can be identified with the complexes of the fractal. More precisely, for a given control level $m_0 \in \mathbb Z,$ one can consider a family of i.i.d.\ nonnegative random variables $\big\{\xi_{\Delta}: \Delta \in \cT_{m_0}\big\}$ and the single-site potential $\overline{W} : \cK^{\langle \infty \rangle} \times \cT_{m_0} \to [0,\infty)$, then define
$$
{\overline{V}}^{\, \omega}(x) := \sum_{\Delta \in \cT_{m_0}} \xi_{\Delta}(\omega) \overline{W}(x,\Delta), \quad x \in \cK^{\langle \infty \rangle}.
$$
Such alloy-type potentials seem better suited to a `birds-eye' view of the random potential configuration: natural structure on
fractals is labeled with complexes, not their vertices.
One can also consider a version of this potential which is constructed for the $\pi_M$-periodized configuration of the sites, where $M \geq m_0$:
$$
{\overline{V}}_{M}^{\, \omega}(x) := \sum_{\Delta_{m_0} \in \cT_{m_0} \atop \Delta_{m_0} \subset \cK^{\langle M \rangle}}
                                       \xi_{\Delta_{m_0}}(\omega) \cdot \sum_{\Delta_{m_0}^{'} \in \pi_M^{-1}(\Delta_{m_0})}\overline{W}(x,\Delta_{m_0}^{'}), \qquad x \in \cK^{\langle \infty \rangle}.
$$
Observe that the potential in \eqref{eq:period_alloy} is exactly the one with
$$
\overline{W}(x,\Delta_{m_0}):=  A_0 \cdot \mathbf 1_{\Delta_{m_0} \setminus \mathcal V^{\langle \infty\rangle}_{m_0}}(x)
$$
and the random variables $\xi_{\Delta_{m_0}}=q_{\Delta_{m_0}}$ induced by the Poisson point process on $\cK^{\langle \infty \rangle}$.
To the best of our knowledge such random fields have not yet been studied on fractals.
}
\end{remark}

\bigskip

We can now conclude the first step of our proof:\
a combination of \eqref{eq:trace-estimate} and \eqref{eq:poten_monot} yields
\begin{equation}\label{eq:trace-estimate-alloy}
 \mathbb L(t) \leq C \mathbb{E}^{\mathbb{Q}} {\rm e}^{-(t-1)\lambda_1^{N,M,{\overline{V}}_{M}^{\, \omega}}}, \quad t>1, \ M \in \mathbb{Z}_+, \ M \geq m_0,
\end{equation}
where $\lambda_1^{N,M,{\overline{V}}_{M}^{\, \omega}}$ is the ground state eigenvalue of the operator
$H_{M}^{N,{\overline{V}}_{M}^{\, \omega}} = \phi(-\cL^M)  + {\overline{V}}_{M}^{\, \omega}$.

 The rest of the proof follows the scheme from \cite{bib:KaPP3} where the integer-lattice alloy-type potentials in Euclidean spaces were considered. These ideas can be adapted to the present setting of alloy-type potentials in which the sites are modeled by fractal complexes. This will be done in the next two steps.

\subsection*{STEP 2. Lower estimate for the ground state eigenvalue $\lambda_1^{N,M,{\overline{V}}_{M}^{\, \omega}}$}

We will need the following classical result.

 \begin{proposition}[{Temple's inequality, \cite[Theorem XIII.5]{bib:RS}}] \label{prop:temple}
 Suppose $H$ is a self-adjoint operator  on a Hilbert space with inner product $\langle \cdot,\cdot \rangle$ such that $\lambda_1:=\inf\sigma(H)$ is an isolated eigenvalue and let $\eta\leq \inf(\sigma(H)\setminus \{\lambda_1\}).$  Then for any $\psi\in\mathcal D(H)$ which satisfies
	\begin{equation}\label{eq:temple-condition}
	\langle \psi, H\psi\rangle <\eta\quad\mbox{ and } \quad  \|\psi\|=1
	\end{equation}
	the following bound holds:
	\begin{equation}\label{eq:temple-ineq}
	\lambda_1\geq \langle \psi, H\psi\rangle - \frac{\langle H\psi, H\psi\rangle-\langle \psi,H\psi\rangle^2}{\eta-\langle \psi, H\psi\rangle}.
	\end{equation}
\end{proposition}

Let
 $\eta=  \lambda_2^M$ be the second eigenvalue of the  operator $\phi(-\cL^M)$).
From \eqref{eq:2.8} we have
$$
\eta = \phi(\mu_2^M)=\phi(L^{-Md_w}\mu_2^1)
$$
and using  our assumption (\ref{B}.b)  we get that there exists $M_1\geq m_0$ such that for $M \in \mathbb Z_+$, $M \geq M_1,$
\begin{align}\label{eq:eigen_scaling}
\widetilde{C}_1\frac{1}{L^{M\alpha}} \leq \eta \leq \widetilde{C}_2\frac{1}{L^{M\alpha}},
\end{align}
with $\widetilde{C}_1 = C_1(\mu_2^1)^{\alpha/d_w}$ and $\widetilde{C}_2 = C_2(\mu_2^1)^{\alpha/d_w}.$ Define:
\begin{equation}\label{per-d0}
D_0 = \frac{\widetilde{C}_1}{4A_0}.
 \end{equation}
 We may and do assume that $M_1$ is large enough so that for $M \geq M_1$ we have
$$
D_0 \leq L^{M \alpha}.
$$
For every $\Delta \in \cT_{m_0}$ and $M \in \mathbb Z_+$, $M \geq M_1$ we set
\begin{equation}\label{eq:q-diminished}
\widetilde{q}_\Delta(\omega):= q_\Delta(\omega)\wedge\frac{D_0}{L^{M\alpha}} = \frac{D_0}{L^{M\alpha}} \cdot q_\Delta(\omega)
\end{equation}
and consider the  potential
\begin{align}
\widetilde {V}_{M}^\omega(x) & = A_0 \cdot \sum_{\Delta \in \cT_{m_0} \atop \Delta \subset \cK^{\langle M \rangle}} \widetilde q_{\Delta}(\omega) \cdot \mathbf 1_{\Delta \setminus \mathcal  V^{\langle \infty\rangle}_{m_0}}(\pi_M(x)) \nonumber \\
& = \frac{A_0 D_0}{L^{M\alpha}} \cdot \sum_{\Delta \in \cT_{m_0} \atop \Delta \subset \cK^{\langle M \rangle}} q_{\Delta}(\omega) \cdot \mathbf 1_{\Delta \setminus \cV^{\langle \infty\rangle}_{m_0}}(\pi_M(x)) \nonumber\\
& = \frac{A_0 D_0}{L^{M\alpha}} \cdot q_{\Delta(\pi_M(x))}(\omega) \cdot \mathbf 1_{\cK^{\langle \infty \rangle} \setminus \mathcal V^{\langle \infty\rangle}_{m_0}}(x) ,\qquad x\in \cK^{\langle \infty \rangle}. \nonumber
\end{align}  Clearly, $\widetilde {V}_{M}^\omega(x) \leq \overline {V}_{M}^\omega(x)$ and, in consequence,
\begin{align} \label{eq:ineq_eigen}
\lambda_1^{N,M,\widetilde {V}_{M}^\omega} \leq \lambda_1^{N,M,\overline {V}_{M}^\omega}.
\end{align}
Therefore it is enough to estimate from below the eigenvalue $\lambda_1^{N,M,\widetilde {V}_{M}^\omega}$ instead of $\lambda_1^{N,M,\overline {V}_{M}^\omega}.$

\smallskip

We have a lemma.
\begin{lemma}\label{lem:l-33} Let the assumptions $\mathbf{(B)}$ and $\mathbf{(W1.a)}$  hold.  Then for any $M \in \mathbb Z_+$, $M \geq M_1$ we have
\begin{equation}\label{eq:l-33-teza}
 \lambda_1^{N,M,\widetilde {V}_{M}^\omega}  \geq \frac{1}{L^{Md}}\left[ \int_{\mathcal{K}^{\left\langle M\right\rangle}} \widetilde{V}^\omega_{M}(x){ \rm d}\mathfrak{m}(x) - \frac{2\int_{\mathcal{K}^{\left\langle M\right\rangle}} (\widetilde{V}^{\omega}_{M}(x))^2{\rm d}\mathfrak{m}(x)}{\widetilde{C}_1L^{-M\alpha}} \right].
\end{equation}
\end{lemma}

\begin{proof} Let us fix  $M \geq M_1.$
We apply Temple's inequality to  $H=H^{N,\widetilde {V}_{M}^\omega}_{M} = \phi(-\cL^M) + \widetilde {V}_{M}^\omega$  and $\eta = \lambda_2^{M}$  (nonrandom).
The spectrum of this operator is discrete and we have that
$$
\eta\leq \lambda_2^{N,M,\widetilde {V}_{M}^\omega}=\inf \left( \sigma(H^{N,\widetilde {V}_{M}^\omega}_{M}) \backslash \left\{\lambda_1^{N,M,\widetilde {V}_{M}^\omega}\right\}\right).
$$
Choose
 $\psi = \psi_1^M \equiv \frac{1}{L^{Md/2}}$ to be the principal normalized eigenfunction  of the operator $\phi(-\cL^M)$.  Thanks to (\ref{eq:2.8}) and (\ref{eq:2.9}), $\phi(-\cL_M)\psi = 0$, and further
$$
\left\langle\psi,H^{N,\widetilde {V}_{M}^\omega}_{M}\psi\right\rangle = \left\langle\psi,\phi(-\cL_M)\psi\right\rangle+ \left\langle\psi,\widetilde{V}^\omega_{M}\psi\right\rangle=
\left\langle\psi,\widetilde{V}^\omega_{M}\psi\right\rangle=
\frac{1}{L^{Md}}\int_{\mathcal{K}^{\left\langle M\right\rangle}} \widetilde{V}^\omega_{M}(x){\rm d}\mathfrak{m}(x).
$$
 Similarly,
$$
\left\langle H^{N,\widetilde {V}_{M}^\omega}_{M}\psi,H^{N,\widetilde {V}_{M}^\omega}_{M}\psi\right\rangle = \frac{1}{L^{Md}}\int_{\mathcal{K}^{\left\langle M\right\rangle}} \big(\widetilde{V}^\omega_{M}(x)\big)^2{\rm d}\mathfrak{m}(x).
$$
From the definition of $\widetilde V_{M}^\omega$ we get
\begin{eqnarray*}
\int_{\mathcal{K}^{\left\langle M\right\rangle}} \widetilde{V}_{M}^\omega(x){\rm d}\mathfrak{m}(x) &=& \frac{A_0D_0}{L^{M\alpha }} \int_{\mathcal{K}^{\left\langle M\right\rangle}} \sum_{\mathcal T_{m_0}\ni\Delta\subset \mathcal{K}^{\langle M\rangle} }{q_\Delta}(\omega)\mathbf 1_{\Delta \setminus \mathcal V^{\langle \infty\rangle}_{m_0}}(x) { \rm d}\mathfrak{m}(x)\\
&\leq & \frac{A_0D_0}{L^{M\alpha }}  L^{Md},
\end{eqnarray*}
and consequently,  by \eqref{per-d0} and \eqref{eq:eigen_scaling},
$$
\left\langle\psi,H^{N,\widetilde {V}_{M}^\omega}_{M}\psi\right\rangle \leq  \frac{A_0D_0}{L^{M\alpha}}  < \frac{\widetilde{C}_1 }{L^{M\alpha}}\leq \eta,
$$
which means that the condition (\ref{eq:temple-condition}) is satisfied.
Moreover, this yields
\begin{equation}\label{eq:temple_den}
\eta - \left\langle\psi,H^{N,\widetilde {V}_{M}^\omega}_{M}\psi\right\rangle \geq
\frac{\widetilde{C}_1}{L^{M\alpha }}-\frac{ A_0D_0}{L^{M\alpha }}  > \frac{\widetilde{C}_1}{2 L^{M\alpha }}.
\end{equation}
Inserting (\ref{eq:temple_den}) into the formula in Temple's inequality (\ref{eq:temple-ineq})  (with $H = H^{N,\widetilde {V}_{M}^\omega}_{M}$),  we obtain
\begin{eqnarray*}
\lambda_1^{N,M,\widetilde {V}_{M}^\omega}  & \geq & \langle \psi, H\psi\rangle -  \frac{\langle H\psi, H\psi\rangle}{\eta - \left\langle\psi,H\psi\right\rangle}\\
& \geq & \frac{1}{L^{Md}}\left[ \int_{\mathcal{K}^{\left\langle M\right\rangle}} \widetilde{V}^{\omega}_{M}(x){\rm d}\mathfrak{m}(x) - \frac{ 2 \int_{\mathcal{K}^{\left\langle M\right\rangle}} (\widetilde{V}^{\omega}_{M}(x))^2{\rm d}\mathfrak{m}(x)}  {\widetilde{C}_1 L^{-M\alpha }} \right].
\end{eqnarray*}
\end{proof}

\medskip

The estimate in Lemma \ref{lem:l-33} will be of use for those realizations of Poisson cloud for which the number of complexes from $\mathcal T_{m_0}$ with at least one Poisson point is sufficiently large. More precisely,
let us fix $\delta \in (0,1)$ (later we show how to choose this value) and define:
\begin{eqnarray*}
\mathcal{A}_{M,\delta} & = & \left\{ \omega: \#\left\{\mathcal T_{m_0}\ni\Delta\subset \mathcal{K}^{\langle M\rangle} : q_\Delta(\omega) = 1\right\} \geq \delta \cdot L^{(M-m_0)d}\right\}.
\end{eqnarray*}
Here $L^{(M-m_0)d}$ is the number of all $m_0$-complexes included in $\mathcal{K}^{\langle M\rangle}$.

We are now ready to prove another auxiliary lemma.

\begin{lemma}\label{lem:34}
Let the assumptions \eqref{B} and \textup{(\ref{W2}.a)}  hold and let $\delta > 0$ be fixed.  Then
for any $M \in \mathbb Z_+$, $M \geq M_1$
and $\omega \in \mathcal{A}_{M,\delta}$ it holds that
$$
\lambda_1^{N,M,{\overline{V}}_{M}^{\, \omega}} \geq \frac{\widetilde{C}_1\delta }{8L^{M\alpha}}.
$$
\end{lemma}
\noindent{Proof. }
From the definition of $\widetilde {V}^\omega_{M}$ we have
\begin{eqnarray*}
\int_{\mathcal{K}^{\left\langle M\right\rangle}} \widetilde {V}^\omega_{M}(x){ \rm d}\mathfrak{m}(x) & = & \frac{A_0D_0}{L^{M\alpha}}\int_{\mathcal{K}^{\left\langle M\right\rangle}} \sum_{\mathcal T_{m_0}\ni\Delta\subset \mathcal{K}^{\langle M\rangle} }{q_\Delta}(\omega)\mathbf 1_{\Delta \setminus \mathcal V^{\langle \infty\rangle}_{m_0}}(x) { \rm d}\mathfrak{m}(x)\\
& = & \frac{A_0D_0}{L^{M\alpha}}\cdot L^{m_0d} \cdot\sum_{\mathcal T_{m_0}\ni\Delta\subset \mathcal{K}^{\langle M\rangle} }{q_\Delta}(\omega).
\end{eqnarray*}
Similarly,
$$
\int_{\mathcal{K}^{\left\langle M\right\rangle}} (\widetilde{V}^{\omega}_{M}(x))^2{ \rm d}\mathfrak{m}(x) = \left(\frac{A_0D_0}{L^{M\alpha}}\right)^2 \cdot L^{m_0d} \cdot \sum_{\mathcal T_{m_0}\ni\Delta\subset \mathcal{K}^{\langle M\rangle} }{q_\Delta}(\omega).
$$
Therefore,  by using \eqref{eq:ineq_eigen}, the estimate from Lemma \ref{lem:l-33}, the assumption that $\omega \in \mathcal{A}_{M,\delta}$ and \eqref{per-d0}, we get
\begin{eqnarray*}
\lambda_1^{N,M,{\overline{V}}_{M}^{\, \omega}} & \geq & \frac{L^{m_0d}}{L^{Md}}\left[ \frac{A_0D_0}{L^{M\alpha}}\sum_{\mathcal T_{m_0} \ni\Delta\subset \mathcal{K}^{\langle M\rangle} }{q_\Delta}(\omega) - \frac{2\left(\frac{A_0D_0}{L^{M\alpha}}\right)^2 \displaystyle\sum_{\mathcal T_{m_0} \ni\Delta\subset \mathcal{K}^{\langle M\rangle} }{q_\Delta}(\omega)}{\widetilde{C}_1 L^{-M\alpha }}\right]\\
& = & \frac{1}{L^{(M-m_0)d}} \cdot \frac{1}{L^{M\alpha}}\cdot A_0D_0 \cdot \left(1-\frac{2A_0D_0}{\widetilde{C}_1 }\right) \cdot\sum_{\mathcal T_{m_0} \ni\Delta\subset \mathcal{K}^{\langle M\rangle} }{q_\Delta}(\omega)\\
& \geq & \frac{\delta A_0D_0}{2L^{M\alpha}} = \frac{\widetilde{C}_1\delta }{8L^{M\alpha}}.
\end{eqnarray*}
This completes the proof.
\hfill$\square$

\medskip

Before we complete the proof of Theorem (\ref{th:main}), we recall
a classical Bernstein-type inequality (see e.g.\ \cite[Lemma 3.5]{bib:KaPP3}).
\begin{proposition}\label{prop:lemma-35}
Let $(\Omega, \mathcal{F}, \mathbb{P})$ be a given probability space and let $S_n: \Omega \rightarrow \mathbb{R}$ be a random variable with the binomial distribution $B(n,p), n \geq 1, p \in (0,1).$ Then, for any $p,\gamma \in (0,1)$ such that $\gamma > p$,
\begin{equation}\label{eq:lemma-35}
\mathbb{P}\left[S_n \geq \gamma n \right] \leq \left(\left(\frac{1-p}{1-\gamma}\right)^{1-\gamma}\left(\frac{p}{\gamma}\right)^\gamma\right)^n.
\end{equation}
\end{proposition}

\medskip

\subsection*{STEP 3. Final estimates.}

From \eqref{eq:trace-estimate-alloy} we have:
\begin{equation}\label{eq:trace-estimate-rewrited}
\mathbb L(t) \leq {C}\mathbb{E}^{\mathbb{Q}}\left[{\rm e}^{-(t-1)\lambda_1^{N,M,{\overline{V}}_{M}^{\, \omega}}}; \mathcal{A}_{M,\delta} \right] + {C}\mathbb{Q}\left[\mathcal{A}^c_{M,\delta}\right], \qquad t > 1, \qquad M \geq  m_0.
\end{equation}
First note that by Lemma \ref{lem:34} we get
$$
\mathbb{E}^{\mathbb{Q}}\left[{\rm e}^{-(t-1)\lambda_1^{N,M,{\overline{V}}_{M}^{\, \omega}}}; \mathcal{A}_{M,\delta} \right] \leq {\rm e}^{-\frac{\widetilde{C}_1\delta }{8L^{M\alpha}}(t-1)} \leq c_1 {\rm e}^{-\frac{c_2 \delta }{L^{M\alpha} }t}, \quad t > 1, \qquad M \geq M_1,
$$
with absolute constants $c_1, c_2 >0$.

To estimate the second term in \eqref{eq:trace-estimate-rewrited}, observe that for $M \geq M_1,$
\begin{eqnarray*}
\mathcal{A}^{c}_{M,\delta} & = & \left\{ \omega: \#\left\{\mathcal T_{m_0}\ni\Delta\subset \mathcal{K}^{\langle M\rangle} : q_\Delta(\omega) = 1 \right\} < \delta \cdot L^{(M-m_0)d}\right\}\\
& = & \left\{ \omega: \#\left\{\mathcal T_{m_0}\ni\Delta\subset \mathcal{K}^{\langle M\rangle} : q_\Delta(\omega) = 0 \right\} \geq(1-\delta)  \cdot L^{(M-m_0)d}\right\}
\end{eqnarray*}
(recall that $L^{(M-m_0)d}$ is the number of $m_0$-complexes included in $\mathcal{K}^{\langle M\rangle}).$
We set
$$
 p_{\nu} := \mathbb Q[q_\Delta=0]=\mathbb Q\big[\text{no Poisson points in} \ \Delta \setminus \cV^{\langle \infty\rangle}_{m_0}\big] = {\rm e}^{-\nu L^{m_0d}}.
$$
Fix $\nu_0>0.$
Since $\nu \mapsto p_{\nu}$ is a decreasing function and $\lim_{\delta \rightarrow 0}\frac{1}{1-\delta}\left(\frac{1}{\delta}\right)^\frac{\delta}{1-\delta} = 1,$ we can also choose $\delta \in (0,1)$ small enough such that for all $\nu \geq \nu_0$
$$
\delta < 1 - p_{\nu_0} \leq 1 - p_{\nu}
\qquad \text{and} \qquad
\frac{1}{1-\delta}\left(\frac{1}{\delta}\right)^\frac{\delta}{1-\delta} \cdot \sqrt{p_{\nu}} \leq \frac{1}{1-\delta}\left(\frac{1}{\delta}\right)^\frac{\delta}{1-\delta} \cdot \sqrt{p_{\nu_0}} < 1.
$$
By applying Proposition  \ref{prop:lemma-35} with $n =  L^{(M-m_0)d}$, $p = p_{\nu}$ and $\gamma = (1-\delta) > p_{\nu}$, we then get
\begin{align*}
\mathbb{Q}\left[\mathcal{A}^{c}_{M,\delta}\right]
& \leq \left[\left(\frac{1-p}{\delta}\right)^\delta \left(\frac{p}{1-\delta}\right)^{1-\delta} \right]^{ L^{(M-m_0)d}}\\
& \leq \left[\frac{1}{1-\delta}\left(\frac{1}{\delta}\right)^\frac{\delta}{1-\delta} \cdot {p}\right]^{(1-\delta) \cdot L^{(M-m_0)d}} \\
& \leq {\rm e}^{\frac{(1-\delta) }{2} L^{(M-m_0)d} \ln p_{\nu}}\\
& =  {\rm e}^{-\nu \frac{1-\delta}{2}L^{Md}}.
\end{align*}
Therefore we obtain that
$$
\mathbb L(t) \leq  c_3({\rm e}^{-\frac{c_4t}{L^{M\alpha}}} +{\rm e}^{-c_5\nu L^{Md}}), \quad t>0, \ \ M \geq M_1, \ \ \nu \geq \nu_0,
$$
with an absolute $c_3>0$ and $c_4=c_4(\nu_0)>0$, $c_5=c_5(\nu_0)>0$.

To minimize this bound over $M,$ we  find $M=M(t)$ in such a manner that the two terms are roughly equal. We choose the  unique $M=M(t)$ for which
\begin{equation}\label{eq:M(t)}
\frac{c_5}{c_4}L^{M(d+\alpha )}\leq \frac{t}{\nu} < \frac{c_5}{c_4}L^{(M+1)(d+\alpha )}.
\end{equation}
For this choice of $M$ we get, with some $c_6>0,$
$$
\frac{c_4t}{L^{M\alpha}} \geq c_5\nu L^{Md}\geq  c_6t^{\frac{d}{d+\alpha}}\nu^{\frac{ \alpha}{d+ \alpha}}
$$
so that finally, for sufficiently large $t$ and some  constants $c_7=c_7(\nu_0),c_8=c_8(\nu_0)>0$
$$
\mathbb L(t) \leq c_7{\rm e}^{-c_8t^{\frac{d}{d+ \alpha}}\nu^{\frac{ \alpha}{d+\alpha}}},
$$
whenever $\nu \geq \nu_0$. This gives the assertion in Theorem \ref{th:main} with $C=c_8$ finishing the proof.

\section{ Lower bound for the Laplace transform and final estimates of the IDS} \label{sec:LS_bounded}
In this section, we find the lower estimate for the Laplace transform $\mathbb L(t)$ of the IDS for a class of Poisson potentials with single-site profiles $W$ of bounded  support, i.e.\ we assume that (\ref{W2}.b) holds. The two bounds on the Laplace transform $\mathbb L(t)$ will then be transformed to the asymptotic behaviour of the IDS via Fukushima's Tauberian theorem of exponential type.

We now take advantage of the fact that the integrated density of states arises also as the limit of finite-volume expressions built for operators  with Dirichlet conditions. More precisely, the Laplace transform $\mathbb L(t)$ of the IDS is given by the formula
\begin{equation}\label{eq:L_t_inty_lower}
\mathbb L(t)=\lim_{M\to\infty}\mathbb E^{\mathbb Q}\mathbb L_M^{D,V^{\omega}}(t)
\end{equation}
where
$$
\mathbb L_M^{D,V^{\omega}}(t) = \frac{1}{L^{Md}}\int_{\mathcal K^{\langle M\rangle}}p(t,x,x)\mathbf{E}_{x,x}^t \left[ {\rm e}^{-\int_0^t V^{\omega}(X_s){\rm d}s}\mathbf{1}_{\{\tau_{\mathcal K^{\langle M\rangle}} > t\}}\right]\mu({\rm d}x)
$$
and $\tau_{A}$ is the first exit time of the process from the set $A,$ i.e.  $\tau_{A} := \inf \{t > 0: X_t \notin A\}.$

\begin{theorem}\label{th:main-2}
Under the assumptions \eqref{B}, \eqref{W1} and \textup{(\ref{W2}.b)}, there exists a constant $C'>0$ such that for every $\nu >0$ we have
\begin{equation*}\label{eq:lower-poiss-1}
\liminf_{t\to\infty} \frac{\log \mathbb L(t)}{t^{\frac{d}{d+\alpha }}}\geq -C'\nu^{\frac{\alpha }{d+\alpha }}.
\end{equation*}
\end{theorem}

\begin{proof}
Before we proceed, we introduce additional notation. For a single complex $\Delta$ (of any size)  we define
\begin{equation}\label{eq:el-delta}
\mathbb L^{D,V^{\omega}}_{\Delta}(t) =\frac{1}{\mathfrak{m}(\Delta)}\int_\Delta p(t,x,x) \mathbf E_{x,x}^t \left[{\rm e}^{-\int_0^t V^\omega (X_s){\rm d}s}\mathbf 1_{\{\tau_{\Delta}>t\}}\right]\mathfrak{m}({\rm d}x).
\end{equation}
Now fix $M\in\mathbb Z_+$, $M \geq M_0$, where $M_0$ comes from our assumption (\ref{W2}.b).  By (\ref{eq:L_t_inty_lower}) we have
$$
\mathbb L(t) = \lim_{k \rightarrow \infty}\mathbb E^{\mathbb Q} \mathbb L_{M+k}^{D,V^{\omega}}(t).
$$
For given $k \geq 1,$ the set $\mathcal{K}^{\langle M + k \rangle}$ contains $N^{k}$ copies of $\mathcal{K}^{\langle M \rangle}$ meeting only through their vertices,  of size $L^M$ each. Denote them $\Delta_1,\ldots,\Delta_{N^k}.$ As $\Delta_i \subset \mathcal{K}^{\langle M + k \rangle}$ we get, since $L^d=N$:\
\begin{eqnarray}\label{eq:dd}
\mathbb E^{\mathbb Q} \mathbb L_{M+k}^{D,V^{\omega}}(t)  & = & \frac{1}{N^{M + k}}\mathbb E^{\mathbb Q}\int_{\mathcal K^{\langle M + k\rangle}}p(t,x,x)\mathbf{E}_{x,x}^t \left[ {\rm e}^{-\int_0^t V^{\omega}(X_s){\rm d}s}\mathbf{1}_{\{\tau_{\mathcal K^{\langle M + k\rangle}} > t\}}\right]\mathfrak{m}({\rm d}x) \nonumber \\
& = & \frac{1}{N^{M + k}}\sum_{i=1}^{N^k}\int_{\Delta_i}p(t,x,x)\mathbb{E^{Q}}\mathbf{E}_{x,x}^t \left[ {\rm e}^{-\int_0^t V^{\omega}(X_s){\rm d}s}\mathbf{1}_{\{\tau_{\mathcal K^{\langle M + k\rangle}} > t\}}\right]\mathfrak{m}({\rm d}x)\nonumber \\
& \geq & \frac{1}{N^{M + k}}\sum_{i=1}^{N^k}\int_{\Delta_i}p(t,x,x)\mathbb{E^{Q}}\mathbf{E}_{x,x}^t \left[ {\rm e}^{-\int_0^t V^{\omega}(X_s){\rm d}s}\mathbf{1}_{\{\tau_{\Delta_i} > t\}}\right]\mathfrak{m}({\rm d}x)\nonumber \\
&=& \frac{1}{N^k}\sum_{i=1}^{N^k} \mathbb E^{\mathbb Q} \mathbb L^{D,V^{\omega}}_{\Delta_i}(t) \nonumber \\
& \geq & \inf_{i}\mathbb{E^{Q}}\mathbb L^{D,V^{\omega}}_{\Delta_i}(t).
\end{eqnarray}
We fix some $i_0$ and let $\mathcal{M}_{i_0}=\{\omega:\, \mbox{no Poisson points fell into } \Delta_{i_0}^{M_0} \}$, where $\Delta_{i_0}^{M_0} $ denotes the $1-$vicinity of $\Delta_{i_0}$ in the metric $d_{M_0}$ (i.e.\ the complex $\Delta_{i_0}$ with copies of $\mathcal K^{\langle M_0\rangle}$ attached at its vertices). In particular,
\begin{equation}\label{eq:calc_m_i-0}
\mathbb{E^{Q}}\mathbb L^{D,V^{\omega}}_{\Delta_{i_0}}(t) \geq \mathbb{E^{Q}}\left[\mathbb L^{D,V^{\omega}}_{\Delta_{i_0}}(t)\mathbf{1}_{\mathcal{M}_{i_0}}\right].
\end{equation}
Let us fix a trajectory $X_s$ of the process starting at $x \in \Delta_{i_0}$ and not leaving the set $\Delta_{i_0}$ up to time $t$.  For $\omega \in \mathcal{M}_{i_0}$ we have that $V^\omega(X_s)=0$, $s \leq t$ (the process stayed in $\Delta_{i_0},$ and no Poisson points fell inside the $M_0-$neighbourhood of this set), and, in consequence, ${\rm e}^{-\int_0^t V^\omega(X_s){\rm d}s}=1$. 
Therefore $\mathbf P_{x,x}^t$ -- almost surely on the set $\{\tau_{\Delta_{i_0}} > t\}$,
\begin{equation}\label{eq:tal}
\mathbb{E^Q}\left[{\rm e}^{-\int_0^t V^\omega(X_s){\rm d}s}\mathbf{1}_{\mathcal{M}_{i_0}} \right]=\mathbb{E^Q}\left[\mathbf{1}_{\mathcal{M}_{i_0}} \right]=\mathbb{Q}[\mathcal{M}_{i_0}],
\end{equation}
  Using (\ref{eq:calc_m_i-0}) we get
\begin{equation}\label{eq:l2}
\mathbb{E^{Q}}\mathbb L^{D,V^{\omega}}_{\Delta_{i_0}}(t)\geq \frac{1}{\mathfrak{m}(\Delta_{i_0})}\int_{\Delta_{i_0}}p(t,x,x)\mathbf{E}_{x,x}^t \left[ \mathbf{1}_{\{\tau_{\Delta_{i_0}} > t\}}\cdot \mathbb{Q}[\mathcal{M}_{i_0}] \right]\mathfrak{m}({\rm d}x).
\end{equation}
As
$$
\mathbb{Q}[\mathcal{M}_{i_0}]=\mathbb{Q}[\mbox{no Poisson points in } \Delta_{i_0}^{M_0}]={\rm e}^{-\nu \mathfrak{m}(\Delta_{i_0}^{M_0})} \geq {\rm e}^{-\nu(L^{Md} +  c_1L^{M_0 d})}
$$
($c_1$ is an absolute constant depending on the geometry of  the fractal only) we obtain
\begin{equation}\label{eq:ddd}
\mathbb{E^{Q}}[\mathbb L^{D,\Delta_{i_0}}(t,\omega)]\geq \left[\frac{1}{\mu(\Delta_{i_0})}\int_{\Delta_{i_0}}p(t,x,x)\mathbf{P}_{x,x}^t \left[ \tau_{\Delta_{i_0}} > t\right]\mathfrak{m}({\rm d}x)\right] \cdot  {\rm e}^{-\nu(L^{Md} +  c_1L^{M_0 d})}.
\end{equation}

The expression $\left[\frac{1}{\mu(\Delta_{i_0})}\int_{\Delta_{i_0}}p(t,x,x)\mathbf{P}_{x,x}^t \left[ \tau_{\Delta_{i_0}} > t\right]\mathfrak{m}({\rm d}x)\right]$ is just the averaged trace of the operator $T_t^{D,\Delta_{i_0}}$ which implies that it is not bigger than ${\rm e}^{-t\lambda_1(\Delta_{i_0})}$ (the notation is self-explaining). From \cite[Theorem 3.4]{bib:Chen-Song} we have $\lambda_1(\Delta_{i_0}) \leq \phi(\lambda_1(\Delta_{i_0}))$, where $\lambda_1(U)$ denotes
the principal Dirichlet eigenvalue of the Brownian motion killed upon exiting $U$. Since Brownian motions on $\Delta_{i_0}$  and $\mathcal{K}^{\langle M \rangle}$ are indistinguishable up to respective exit times of $\Delta_{i_0}, \mathcal{K}^{\langle M \rangle},$ one has
$\lambda_1(\Delta_{i_0}) = \lambda_1(\mathcal{K}^{\langle M \rangle})$ and from the Brownian scaling we further have  $\lambda_1(\mathcal{K}^{\langle M \rangle}) = \lambda_1(L^M\mathcal{K}^{\langle 0 \rangle}) = L^{-Md_w}\lambda_1(\mathcal{K}^{\langle 0 \rangle})$. Using (\ref{eq:dd}), (\ref{eq:ddd}) and the scaling of the principal eigenvalues  we obtain that for any $k$
\begin{equation}\label{eq:L-M-k}
\mathbb E^{\mathbb Q}[\mathbb L_{M+k}^{D}(t,\omega)] \geq  \exp \left\{ -t\phi\left(\frac{1}{L^{Md_w}}\lambda_1^{BM}(\mathcal{K}^{\langle 0 \rangle})\right) -\nu(L^{Md} +  c_1L^{M_0 d}) \right\}.
\end{equation}
By using our assumption (\ref{B}.b)we obtain for sufficiently large $M$
\begin{equation}\label{eq:scaling-uses}
\phi\left(\frac{1}{L^{Md_w}}\lambda_1(\mathcal{K}^{\langle 0 \rangle})\right) \leq C_2\cdot L^{-M\alpha }\left( \lambda_1(\mathcal{K}^{\langle 0 \rangle})\right)^{\alpha/d_w}=: c_2 L^{-M\alpha}.
\end{equation}
It then follows from (\ref{eq:L-M-k}) that
\begin{equation}\label{eq:L-M-k-2}
\mathbb E^{\mathbb Q}[\mathbb L_{M+k}^{D}(t,\omega)] \geq  \exp \left\{ -c_3tL^{-M\alpha } -  (1+c_3)  \nu L^{Md} \right\},
\end{equation}
for $M$'s large enough.
This bound is independent of $k,$ therefore passing with $k$ to infinity in (\ref{eq:L_t_inty_lower}) we also obtain
\begin{equation}\label{eq:nowe1}
\mathbb L(t)\geq  \exp \left\{ -c_3tL^{-M\alpha} -c_4\nu L^{Md} \right\},
\end{equation}
  as long as $M$ is sufficiently large.  To conclude, for sufficiently large $t$ choose $M = M(t)$ to be the unique integer satisfying
\begin{equation}\label{eq:lower-Mt}
L^M \leq \left(\frac{t}{\nu}\right)^{\frac{1}{d + \alpha }} < L^{M + 1}.
\end{equation}
Inserting this value of  $M$  into (\ref{eq:nowe1}) and using the inequalities in (\ref{eq:lower-Mt}), after some elementary calculations we end up with
$$
\liminf_{t\to\infty}\frac{\log \mathbb L(t)}{t^{\frac{d}{d + \alpha }}\nu^{\frac{\alpha}{d + \alpha }}}\geq -C',
$$
with a positive constant $C'$. The proof is complete.
\end{proof}

We are now ready to give the argument for  the main result. 
\begin{proof}[Proof of Theorem \ref{th:main_IDS}]
If follows from respective estimates in Theorems \ref{th:main} and \ref{th:main-2} by means of the exponential Tauberian theorem of Fukushima \cite[Theorem 2.1]{bib:F}. 
\end{proof}

\appendix 
\section{Proof of the existence of the IDS} \label{sec:appA}
In this appendix, we continue the discussion from Section \ref{sec:IDS_ex}. We outline the argument leading to the proof of Theorem \ref{thm:IDS} and provide precise references.

Recall that for $M \in \mathbb{Z}$, $V_M^\omega$ denotes the periodized random fields introduced in \eqref{eq:potendef-per}. 
The advantage of using $V_M^\omega$ instead of $V^\omega$, and $X^M$ instead of $X$, 
is that in this setting the expected values of the Laplace transforms are monotone  (Lemma \ref{lem:monoto}), 
which facilitates  their convergence  to a finite limit $\mathbb L(t)$ (Theorem~\ref{thm:convergence_with_star}). 

In the remaining part of the proof, we justify the convergence of 
$\mathbb{E}^\mathbb{Q}  \mathbb L_M^{D,V^\omega}(t)$, 
$\mathbb{E}^\mathbb{Q}  \mathbb L_M^{D,V^\omega_M}(t)$, 
and $\mathbb{E}^\mathbb{Q}  \mathbb L_M^{N,V^\omega}(t)$ 
to a common limit. Finally, we establish the almost sure convergence of 
$\mathbb L^{D,V^\omega}_M(t)$ and $\mathbb L^{N,V^\omega}_M(t)$ 
to the nonrandom limit $\mathbb L(t)$.

Our first step will be to show that once the path of the initial subordinate Brownian motion $X$ on $\cK^{\langle \infty\rangle}$ and $t>0$ are fixed, then the expectations $\mathbb{E}^{\mathbb{Q}} {\rm e}^{-\int_0^t V^{\omega}_M(\pi_M(X_s)){\rm d}s} = \mathbb{E}^{\mathbb{Q}} {\rm e}^{-\int_0^t V^{\omega}_M(X_s^M){\rm d}s}$ are monotone in $M.$

\begin{lemma}
\label{lem:monoto}
For every $t>0$ and  $M \in \mathbb{Z}_{+}$ we have
\begin{equation}
\label{eq:monoto}
\mathbb{E}^{\mathbb{Q}} {\rm e}^{-\int_0^t V^{\omega}_{M+1}(\pi_{M+1}(X_s)){\rm d}s} \leq \mathbb{E}^{\mathbb{Q}} {\rm e}^{-\int_0^t V^{\omega}_M(\pi_M(X_s)){\rm d}s} \ .
\end{equation}
\end{lemma}

\begin{proof}
The proof follows analogously to that of  inequality (3.8) in the proof of  \cite[Theorem 3.1]{bib:KaPP2}. For $y \in \cK^{\langle M \rangle} \backslash \cV^{\langle M \rangle}_M$ we have $\sum_{i=1}^{N} \pi_{M+1}^{-1}(\pi_{\Delta_{M,i}}(y)) = \pi_M^{-1}(y)$ and this sum is disjoint. In the proof of mentioned theorem we had $N=3$.
\end{proof}

The next step is to apply the monotonicity argument from Lemma \ref{lem:monoto} above to prove that the expectations $\mathbb{E}^{\mathbb{Q}} \mathbb L_{M}^{N,V^\omega_M} (t)$ converges, for every fixed $t>0$, to a finite limit $\mathbb L(t)$.
Here $\mathbb L_{M}^{N,V^\omega_M} (t,\omega)$ is the Laplace transform of the measure $\Lambda_N^{M,V^\omega_M},$ as in \eqref{eq:lmn}.

\begin{theorem} \label{thm:convergence_with_star}
Let $\cK^{\langle \infty \rangle}$ be an USNF with the GLP and let the assumptions \eqref{B} and \eqref{W1} hold. Then, for every $t>0$,
$$\mathbb{E}^{\mathbb{Q}}\mathbb L_{M}^{N,V^\omega_M}(t,\omega) \searrow \mathbb L(t) \quad  \text{as} \quad M \to \infty.$$
\end{theorem}

\begin{proof}
The proof follows closely the reasoning from the second part of the proof of \cite[Theorem 3.1]{bib:KaPP2} and uses equations from \cite[Lemma 2.2]{bib:HB-KK-MO-KPP}.
\end{proof}

 Next we show that the limit points of $\mathbb {E^Q}\mathbb L_{M}^{D,V^\omega}(t)$, $\mathbb {E^Q}\mathbb L_M^{D,V_M^\omega} (t),$ and $\mathbb {E^Q}\mathbb L_M^{N,V^\omega}(t)$ coincide with the limit points of $\mathbb {E^Q}\mathbb L_{M}^{N,V^\omega_M}(t)$  as $M \to \infty$. Finally, we also prove the lemma on variances of $\mathbb L_M^{D,V^\omega} (t)$ and $\mathbb L_M^{N,V^\omega} (t)$ which will allow us to establish the vague convergence of the empirical measures defined in \eqref{eq:lmd} and \eqref{eq:lmn}.

Proofs of Lemmas \ref{lem:commonlimit} and \ref{lem:variances} given below are very technical and they follow the steps and ideas from the proofs of \cite[Proposition 3.1 and Lemmas 3.1-3.2]{bib:KaPP2}. The main difference here is that the state space is now a general USNF. This causes some extra geometric issues, which are solved by using the graph distance (recall that the geodesic metric may not be defined at all) and the comparison principle from \cite[Lemma A.2]{bib:KOP}. Therefore, we will follow only the main steps of the proofs which are affected by these changes, focusing on the most critical differences.

\begin{lemma}
\label{lem:commonlimit}
Fix $t>0$. \begin{itemize}
\item[(a)] We have
\begin{equation*}
\sum_{M=1}^{\infty} \mathbb{E}^{\mathbb{Q}} \left(\mathbb L_M^{D,V^\omega} (t) - \mathbb L_M^{N,V^\omega}(t) \right)^2 < \infty;
\end{equation*}
in particular,
\begin{equation*}
\lim_{M \to \infty} \mathbb{E}^{\mathbb{Q}} \left(\mathbb L_M^{D,V^\omega} (t) - \mathbb L_M^{N,V^\omega} (t) \right)^2 = 0 .
\end{equation*}
\item[(b)] We have
\begin{equation*}
\lim_{M \to \infty} \mathbb{E}^{\mathbb{Q}} \left(\mathbb L_{M}^{D,V_M^\omega}(t)  - \mathbb L_{M}^{N,V^\omega_M} (t) \right) = 0 \quad
\text{and} \quad
\lim_{M \to \infty} \mathbb{E}^{\mathbb{Q}} \left( \mathbb L_M^{D,V^\omega} (t) - \mathbb L_{M}^{D,V_M^\omega} (t)\right) = 0 .
\end{equation*}
\end{itemize}
\end{lemma}

\begin{proof}
The proof follows analogously to the proofs of \cite[Proposition 3.1 and Lemma 3.1]{bib:KaPP2}. The key difference is that we cannot use the geodesic metric in our estimates (see the discussion in the last lines of the first column on p.\ 4 in \cite{bib:KOP}).

Denote the vertices from $\cV_M^{\langle M \rangle}$ by $v_i$, $1\leq i\leq k$, and let $\Delta_{\left\lfloor M/2 \right\rfloor,v_i} \subset \cK^{\langle M \rangle}$, be the $\left\lfloor M/2 \right\rfloor$-complex attached to $v_i$.

If the process starts from $x \in \cD_M:=\cK^{\langle M \rangle} \backslash \bigcup_{i=1}^{k} \Delta_{\left\lfloor M/2 \right\rfloor,v_i}$, then, using \cite[Lemma A.2]{bib:KOP}, we have
\begin{equation*}
\left\{t\geq \tau_{\cK^{\langle M \rangle}}\right\} \subseteq \left\{ \sup_{0<s\leq t} d_{\left\lfloor M /2 \right\rfloor}(x,X_s) >2\right\} \subseteq \left\{ \sup_{0<s\leq t} |x-X_s| > c_{1} L^{M/2}\right\},
\end{equation*}
with a constant $c_1$, independent of $M$. We also have
\begin{equation*}
\left\{\sup_{0<s\leq t} |x-X_s| > c_{1} L^{M/2} \right\} \subseteq \left\{\sup_{0<s\leq t/2} |x-X_s| > c_{1} L^{M/2} \right\} \cup \left\{\sup_{t/2<s\leq t} |x-X_s| > c_{1} L^{M/2} \right\}.
\end{equation*}

Recall that $\mathfrak{m} \left( \cK^{\langle M \rangle} \backslash \cD_M \right) = k N^{\left\lfloor M/2 \right\rfloor}$. Using this set instead of the geodesic ball\linebreak $B(0,2^M-2^{(M/2)})$ we can follow the lines of the proofs of \cite[Proposition 3.1 and Lemma 3.1]{bib:KaPP2}.
\end{proof}

We get the following direct corollary from Theorem \ref{thm:convergence_with_star} and Lemma \ref{lem:commonlimit}.

\begin{corollary}
\label{cor:final}
For every $t>0$, $\mathbb{E}^{\mathbb{Q}}\mathbb L_M^{D,V^\omega}(t)$, $\mathbb{E}^{\mathbb{Q}}\mathbb L_{M}^{D,V_M^\omega}(t)$ and $\mathbb{E}^{\mathbb{Q}}\mathbb L_M^{N,V^\omega}(t)$ are also convergent as $M \to \infty$ to $\mathbb L(t)$ identified in Theorem \ref{thm:convergence_with_star}.
\end{corollary}

The last  step is to show that the series of variances of the random variables $\mathbb L_M^{D,V^\omega}(t)$ and $\mathbb L_M^{N,V^\omega}(t)$ are convergent.

\begin{lemma} \label{lem:variances}
For every $t>0$ we have
\begin{equation}
\label{eq:dvarlemma}
\sum_{M=1}^{\infty} \mathbb{E}^{\mathbb{Q}} \left[ \mathbb L^{D,V^\omega}_M (t) - \mathbb{E}^{\mathbb{Q}}\mathbb L_M^{D,V^\omega}(t)\right]^2 < \infty
\end{equation}
and
\begin{equation}
\label{eq:nvarlemma}
\sum_{M=1}^{\infty} \mathbb{E}^{\mathbb{Q}} \left[ \mathbb L^{N,V^\omega}_M (t) - \mathbb{E}^{\mathbb{Q}}\mathbb L_M^{N,V^\omega}(t)\right]^2 < \infty.
\end{equation}
\end{lemma}

\begin{proof}
For $m \in \mathbb{Z}_{+}$ we set
\begin{equation}
V^{\omega,m}(x):= \int_{\mathcal C_m(x)} W(x,y) \mu^{\omega}(\textrm{d}y)
\end{equation}
and
\begin{equation}
\widetilde{V}^{\omega,m}(x):= \int_{\mathcal C_m(x)^c} W(x,y) \mu^{\omega}(\textrm{d}y).
\end{equation}
Recall that for $x \in \cK^{\langle \infty \rangle} \backslash \cV^{\langle \infty \rangle}_m$, $\mathcal C_m(x)$ is the unique $m$-complex containing $x$; for $x \in \cV^{\langle \infty \rangle}_m$ it is a sum of $m$-complexes attached to $x$ (there are $\textrm{rank}_m(x) \in \{1,2,3\}$ of them).

Using the expressions above as definitions of $V^{\omega,m}(x)$ and $\widetilde{V}^{\omega,m}(x)$ we  can argue as in  the proof of \cite[Lemma 3.2]{bib:KaPP2}.
\end{proof}

\begin{proof}[Proof of Theorem \ref{thm:IDS}]
Having proven Theorem \ref{thm:convergence_with_star} and Lemma \ref{lem:variances}, we can just follow the lines of the proof of \cite[Theorem 3.2]{bib:KaPP2}:\ by a Borel-Cantelli argument we get that $\mathbb Q-$almost surely, for all rational $t$'s   it holds
$\mathbb L_M^{D,V^\omega} (t)\to  L(t).$ We extend this to all $t>0$ by continuity. In particular, $\mathbb L_M^{D,V^\omega} (1)\to \mathbb L(1),$ yielding that almost surely the measures ${\rm e}^{-\lambda} \Lambda_M^{D,V^\omega}({\rm d}\lambda),$ and consequently also  $\Lambda_M^{D,V^\omega}({\rm d}\lambda)$, are finite. As every  sequence of finite measures on $\mathbb R_+$ is vaguely relatively compact, the almost sure vague convergence $\Lambda_M^{D,V^\omega}({\rm d}\lambda)\to \Lambda({\rm d}\lambda)$  follows. For $\Lambda_M^{N,V^\omega}({\rm d}\lambda)$ the argument can be repeated verbatim.
\end{proof}

\end{document}